\documentclass{amsart}
\usepackage{amsmath,amsthm,amssymb,amsfonts,color,mathtools}

\usepackage{graphicx}
\usepackage{tikz}

\usepackage{epsfig}
\usepackage[pagebackref, colorlinks=true]{hyperref}

\hypersetup{urlcolor=blue, citecolor=blue}

\def\doi#1{   {\href{http://dx.doi.org/#1}
		{{\mdseries\ttfamily DOI}}}}

\usepackage{graphics}

\newcommand{\R}{\mathbb{R}}
\newcommand{\N}{\mathbb{N}}

\newcommand{\half}{\frac{1}{2}}

\newcommand{\beeq}{\begin{equation}}\newcommand{\eneq}{\end{equation}}

\newcommand{\supp}{\mathrm{supp} \ }

\newenvironment{prf}{\noindent {\bf Proof.} }{\endprf\par}
\def \endprf{\hfill  {\vrule height6pt width6pt depth0pt}\medskip}

\def\<{\langle}             \def\>{\rangle}
\def\({\left(}                 \def\){\right)}
\newtheorem{theorem}{Theorem}[section]

\newtheorem{lemma}[theorem]{Lemma}
\newtheorem{proposition}[theorem]{Proposition}

\newtheorem{definition}[theorem]{Definition}
\newtheorem{remark}[theorem]{Remark}

\newcommand{\ep}{\varepsilon}

\title[Existence of global solutions to NLWs on exterior domains]
{Criteria of the existence of global solutions to semilinear wave equations with first-order derivatives on exterior domains
}

\author{Kerun Shao}
\address{School of Mathematical Sciences\\ Zhejiang University\\Hangzhou 310058\\ P. R. China}\email{shaokr@163.com}



\date{\today}

\begin{document}
\begin{abstract}
	We study the existence of global solutions to semilinear wave equations on exterior domains $\R^n\setminus\mathcal{K}$, $n\geq2$, with small initial data and nonlinear terms $F(\partial u)$ where $F\in C^\kappa$ and $\partial^{\leq\kappa}F(0)=0$. If $n\geq2$ and $\kappa>n/2$, criteria of the existence of a global solution for general initial data  are provided, except for non-empty obstacles $\mathcal{K}$ when $n=2$. For $n\geq3$ and $1\leq\kappa\leq n/2$, we verify the criteria for radial solutions provided obstacles $\mathcal{K}$ are closed balls centered at origin. These criteria are established by local energy estimates and the weighted Sobolev embedding including trace estimates. Meanwhile, for the sample choice of the nonlinear term and initial data, sharp estimates of lifespan are obtained. 
\end{abstract}
\maketitle
\section{Introduction}
Let $n\geq2$ and $\mathcal{K}\subset\R^n$ be a smooth, compact and nontrapping obstacle. We denote $M=\R^n\setminus\mathcal{K}$, hence $M=\R^n$ when $\mathcal{K}=\emptyset$, and consider the following semilinear wave equation
\begin{equation}\label{eq-critnonlin-M}
	\left\{\begin{aligned}
		&\Box u(t,x)\coloneq\partial_{t}^2u-\Delta u=F(\partial u)&&, \ (t,x)\in (0,T)\times M, \ \\
		&u(t,x)=0&&, \ x\in \partial M,  \ t>0 , \\
		&u(0,x)=f(x), \ u_t(0,x)=g(x) &&, \ x\in M. 
	\end{aligned}
	\right.
\end{equation}
Here, $\partial u=(\partial_t u, \partial_1 u, \cdots, \partial_n u)$, the real-valued function 
\begin{equation*}
	F\in C^\kappa(\{\mathbf{q}\in\R^{n+1}:|\mathbf{q}|\leq1\}),\  \kappa\geq1,  \ \partial_\mathbf{q}^{\leq\kappa} F(0)=0,
\end{equation*}
and real-valued initial data $(f,g)$ satisfy some compatibility conditions when $\mathcal{K}$ is non-empty. By translation and scaling, when $\mathcal{K}$ is non-empty, we can assume that the origin is in the interior of $\mathcal{K}$ and $\mathcal{K}\subset \overline{B_1}$. Here, $B_1\subset\R^n$ is the unit open ball center at the origin. Denote $\rho(\tau)=\sup_{|\mathbf{q}|\leq\tau}|\partial_\mathbf{q}^{\kappa}F(\mathbf{q})|$, $0\leq\tau\leq1$. When there is no danger of misunderstanding, we will omit $\mathbf{q}$ in $\partial_\mathbf{q}$ hereby.

The purpose of this paper is to show how the combination of local energy estimates and the control of the $\|\partial u\|_\infty$ leads to the criterion of global existence theorems for small amplitude initial data and nonlinear terms only involving first-order derivatives. The problem \eqref{eq-critnonlin-M} is a generalized version of Glassey conjecture in the exterior domain; see Hidano-Wang-Yokoyama \cite{MR2980460} and the references therein. 
For spatial dimension $n\geq3$, we prove the global existence of any small amplitude solutions to \eqref{eq-critnonlin-M}, if $\int_{0}^{1}\rho(\tau)\tau^{\kappa-3}<\infty$ for $n=3$ and $\kappa>\frac{n}{2}$ for $n\geq4$. Also, for sample choices of the initial data and nonlinear terms, the solutions will blow up at finite time, when $n=3$, $\kappa=2$, and $\int_{0}^{1}\rho(\tau)\tau^{-1} d\tau=\infty$. 
When $1\leq\kappa\leq \frac{n}{2}$, for spatial dimension $n\geq3$, the current technology can only be applied for the radial cases and non-empty obstacles $\mathcal{K}$, that is, as long as $\mathcal{K}$ is a closed ball centered at origin, we can obtain the global existence of any small amplitude radial solution for any nonlinear term $F(\partial_tu,\partial_r u)$ satisfying $\kappa=1$ and $\int_{0}^{1}\rho(\tau)\tau^{-\frac{2}{n-1}-1}d\tau<\infty$, or $\kappa\geq2$. Again, for sample choices of initial data and nonlinear terms, the solutions cannot extend to infinity, when $\kappa=1$ and $\int_{0}^{1}\rho(\tau)\tau^{-\frac{2}{n-1}-1}d\tau=\infty$.
For spatial dimension $2$, due to the lack of local energy estimates, similar results are only established for $M=\R^2$ and $\kappa\geq2$. Given any small amplitude initial data, global solutions exist when $\int_{0}^{1}\rho(\tau)\tau^{\kappa-4}d\tau<\infty$. For sample choices of initial data and nonlinear terms, the solutions will blow up in finite time, when $\int_{0}^{1}\rho(\tau)\tau^{\kappa-4}d\tau=\infty$. On the whole, the existence of a global solution to \eqref{eq-critnonlin-M} with small initial data basically depends on whether the integral $\int_{0}^{1}\rho(\tau)\tau^{\kappa-p_c(n)-1}d\tau$, $p_c(n)=1+2/(n-1)$, converges or not. Meanwhile, the sharp estimates of the lifespan, the maximal existence time of the solution to \eqref{eq-critnonlin-M}, are also provided for the blow-up solutions for sample choices of the initial data and nonlinear terms.

Let us recall the history of the Glassey conjecture. For Cauchy problem
\begin{equation}\label{eq-intro-classical-Glassey}
	\left\{\begin{aligned}
		&\Box u(t,x)=a|\partial_t u|^p+b|\nabla u|^p&&, \ (t,x)\in (0,T)\times \R^n, \ \\
		&u(0,x)=\ep \phi(x), \ u_t(0,x)=\ep \psi(x) &&, \ x\in \R^n, 
	\end{aligned}
	\right.
\end{equation}
with $p>1$ and $a,b\in\R$, it is conjectured that the critical power for the global existence v.s. blow-up is $p_c(n)$ in Glassey's mathematical review \cite{MR0711440Review}; see also Schaeffer \cite{MR0829595} and Rammaha \cite{MR0879355}. For $a\neq0$, $b=0$, the global existence part of the conjecture is verified for general initial data in dimension $2,3$ by Hidano-Tsutaya \cite{MR1386769} and Tzvetkov \cite{MR1637692}, independently. For $n\geq4$, the conjecture is only demonstrated for radial initial data and $p\in(p_c(n),1+2/(n-2))$; see Hidano-Wang-Yokoyama \cite{MR2980460}. For $a>0$ and sample choices of the initial data, the blow-up result is verified by Rammaha \cite{MR0879355} for all spatial dimension $n\geq2$ expect for critical cases, $p=p_c(n)$, in the even dimension. Later, Zhou \cite{MR1845748} shows the blow-up for all dimension $n\geq1$ with $p\in(1,p_c(n)]$ where $p_c(1)=\infty$, together with upper bound estimate on the lifespan for $p\in(1,p_c(n)]$, which is 
\begin{equation}\label{eq-intro-power-upperbound}
	\varlimsup_{\ep\rightarrow 0^+}	\int_0^T [(1+t)^{-(n-1)/2}\ep]^{p-1} dt<\infty, \ \forall p\in (1, p_c(n)].
\end{equation}
On the other hand, deduced from the well-posed theory for $p\in(1,p_c(n)]$, the lower bound on the lifespan is also obtained for all spatial dimensions, that is,
\begin{equation}\label{eq-intro-power-lowerbound}
	\varliminf_{\ep\rightarrow 0^+}	\int_0^T [(1+t)^{-(n-1)/2}\ep]^{p-1} dt>0, \forall p\in (1, p_c(n)];
\end{equation}
see Hidano-Wang-Yokoyama \cite{MR2980460} for $n\geq2$ (see also Fang-Wang \cite{MR3169752} for the critical case in dimension two) and Kitamura-Morisawa-Takamura \cite{MR4558951} for dimension one. Hence, for nonlinear terms $F(\partial u)=|\partial_t u|^p$, the estimates of the lifespan \eqref{eq-intro-power-upperbound} and \eqref{eq-intro-power-lowerbound} are sharp.

As for $a=0,b\neq0$, the results for the global existence are the same as above, as well as the lower bound estimates for $p\in (1,p_c(n)]$. Hence, when $b>0$, one may expect the powers for blow-up results are also $(1,p_c(n)]$ and the lifespan estimates \eqref{eq-intro-power-upperbound} and \eqref{eq-intro-power-lowerbound} are sharp. For dimension one, the sharp lifespan estimates are verified in Sasaki-Takamatsu-Takamura \cite{MR4643157}. For dimension $n=3$, the blow-up result is first established in Sideris \cite{MR0711440} for $p=p_c(3)=2$, as well as the upper bound estimate $\varlimsup_{\ep\rightarrow 0^+}\ln(T(\ep))\ep^2<\infty$, which does not match the well-known lower bound verified in John-Klainerman \cite{MR745325}, $\varliminf_{\ep\rightarrow 0^+}\ln(T(\ep))\ep>0$. For $n=2$, Schaeffer \cite{MR0829595} obtains the blow-up result for $p=p_c(2)=3$. For $n\geq4$, Rammaha shows the blow-up result in \cite{MR0879355} for all of the critical and subcritical powers excluding the critical cases for even dimensions and studies the upper bound for $n=2,3$ and $p=2$ in \cite{MR1338808}. In Takamura, Wang and the author's recent work \cite{MR4819613}, we provide a uniform way to deduce the sharp upper bound estimates of the lifespan for all dimension $n\geq2$ and $p\in(1,p_c(n)]$, that is,
\begin{equation*}
	\left\{
	\begin{aligned}
		&\varlimsup_{\ep\rightarrow 0^+}\ln(T(\ep))\ep^{p-1}<\infty ,\ p=p_c(n);\\
		&\varlimsup_{\ep\rightarrow 0^+}T(\ep)\ep^{\frac{2(p-1)}{2-(n-1)(p-1)}}<\infty , \ p\in(1,p_c(n)).
	\end{aligned}
	\right.
\end{equation*} 

There are also similar results when the problem \eqref{eq-intro-classical-Glassey} is considered on the asymptotically flat Lorentzian manifolds or the exterior domain; see Wang \cite{MR3378835,MR3338309}.

Since the critical power $p_c(n)$ is the borderline of the existence of the global solution with power-type nonlinearities. To determine the threshold nature of the nonlinear term, one may study the problem
\begin{equation}\label{eq-intro-threshold-powertype}
	\left\{\begin{aligned}
		&\Box u(t,x)=|\partial_t u|^{p_c(n)}\mu(|\partial_t u|)&&, \ (t,x)\in (0,T)\times \R^n, \ \\
		&u(0,x)=\ep \phi(x), \ u_t(0,x)=\ep \psi(x) &&, \ x\in \R^n, 
	\end{aligned}
	\right.
\end{equation}
where the function $\mu:[0,\infty)\rightarrow[0,\infty)$ is a non-decreasing sufficiently smooth function with $\mu(0)=0$. The blow-up is established by Chen-Palmieri \cite[Theorem 2.1]{2306.11478} for all classical solutions to \eqref{eq-intro-threshold-powertype} under the assumptions that the function $\tilde{F}:s\in\R\rightarrow|s|^{p_c(n)}\mu(|s|)$ is convex, that $\int_{0}^{1}\mu(s)s^{-1}ds=\infty$, and that initial data $\phi,\psi\in C_c^\infty$ with $\int_{\R^n}\psi(x)>0$. Also, they provide the upper bound estimates for the lifespan, that is, for all sufficiently small $\ep$, there exist constants $K_1$, $K_2$, and $K_3$ such that the lifespan $T(\ep)$ of the solution to \eqref{eq-intro-threshold-powertype} satisfies
\begin{equation*}
	T(\ep)\leq K_3\ep^{\frac{2}{n-1}}\left[\mathcal{H}^{inv}(K_1\ep^{-\frac{2}{n-1}}+\mathcal{H}(K_2\ep))\right]^{-\frac{2}{n-1}},
\end{equation*}
where the function $\mathcal{H}:s\in(0,1]\in\rightarrow \int_{s}^{1}\mu(\tau)\tau^{-1}d\tau$ and $\mathcal{H}^{inv}$ is the inverse function of $\mathcal{H}$. They also obtain results for the existence part when $n=3$ and initial data $\phi,\psi\in C_c^\infty$ are radial. If $\int_{0}^{1}\mu(s)s^{-1}ds<\infty$, there exists a unique radial global classical solution $u$ to \eqref{eq-intro-threshold-powertype} for all sufficiently small $\ep$, provided that $\mu\in C^2((0,\infty))$ satisfying $|\mu^{(i)}(\tau)|\lesssim\tau^{-i}\mu(\tau)$, $i=1,2$. When $\int_{0}^{1}\mu(s)s^{-1}ds=\infty$, the lifespan of the radial classical solution to \eqref{eq-intro-threshold-powertype} has the lower bound
\begin{equation*}
	T(\ep)\geq K_6\ep\left[\mathcal{H}^{inv}(K_4\ep^{-1}+\mathcal{H}(K_5\ep))\right]^{-1},
\end{equation*}
where $K_4$, $K_5$, and $K_6$ are some positive constants; see \cite[Theorem 2.2 and Propsition 2.2]{2306.11478}.

With these results in hand, it is natural to ask what is the criterion for general non-linear terms $F(\partial u)$. Since $\partial^{\leq\kappa} F(0)=0$, it follows that in some neighborhood of the origin $|F(\mathbf{q})|\lesssim|\mathbf{q}|^\kappa\rho(|\mathbf{q}|)$. Heuristically, if the obstacle $K$ is empty, the solutions to \eqref{eq-critnonlin-M} could behave like free waves with energy of size $\ep$ and decay rate $(n-1)/2$, for the time interval $[0,T)$, when
\begin{equation}\label{eq-intro-apriori-calculate}
	\int_{0}^{T}[\<t\>^{-\frac{n-1}{2}}\ep]^{\kappa-1}\rho\left(K_7\<t\>^{-\frac{n-1}{2}}\ep\right)dt\ll 1.
\end{equation}
Here, $K_7$ is some positive constant from the weighted Sobolev embedding and $\<t\>=\sqrt{1+t^2}$.
With simple calculation, the main term of the left hand side of \eqref{eq-intro-apriori-calculate} is
\begin{equation*}
	\ep^{\kappa-1}\int_{\<T\>^{-\frac{n-1}{2}}}^{1}\frac{\rho(K_7\ep\tau)}{\tau}\tau^{\kappa-p_c(n)}d\tau=K_7^{p_c(n)-\kappa}\ep^{\frac{2}{n-1}}\int_{K_7\ep\<T\>^{-\frac{n-1}{2}}}^{K_7\ep}\frac{\rho(\tau)}{\tau}\tau^{\kappa-p_c(n)}d\tau,
\end{equation*}
which gives rise to the following expected sharp estimate for the lifespan, provided $\int_{0}^{1}\rho(\tau)\tau^{\kappa-p_c(n)-1}d\tau=\infty$,
\begin{equation*}
	T\ep^{-\frac{2}{n-1}}\left[\mathcal{H}_{\kappa,n}^{inv}(K_8\ep^{-\frac{2}{n-1}}+\mathcal{H}_{\kappa,n}(K_7\ep))\right]^{\frac{2}{n-1}}\sim1,
\end{equation*}
where $K_8$ is some positive constant and $\mathcal{H}_{\kappa,n}$ is
\begin{equation*}
	\mathcal{H}_{\kappa,n}:s\in(0,1]\rightarrow\int_{s}^{1}\frac{\rho(\tau)}{\tau}\tau^{\kappa-p_c(n)}d\tau
\end{equation*}
and $\mathcal{H}_{\kappa,n}^{inv}$ is the inverse function of $\mathcal{H}_{\kappa,n}$.

Now, we can state our results.

\begin{theorem}\label{thm-3dim}
	Let $n=3$, $\kappa= 2$, and $\mathcal{K}$ be smooth, compact and nontrapping obstacles. Provided $\int_{0}^{1}\rho(\tau)\tau^{-1} d\tau<\infty$, there exists a small positive constant $\ep_1$, such that, the problem \eqref{eq-critnonlin-M} has a unique global solution satisfying
	\begin{equation*}
		\partial_t^{i} u\in C([0,\infty);H_D^{3-i}(M)) , \ 0\leq i\leq2,
	\end{equation*}
	whenever the initial data satisfying the compatibility of order $3$, and
	\begin{equation*}
		\sum_{|\alpha|\leq2}\|(\nabla,\Omega)^\alpha(\nabla f,g)\|_{L^2(M)}=\ep\leq\ep_1, \ \|f\|_{L^2(M)}<\infty.
	\end{equation*}
	Otherwise, if $\int_{0}^{1}\rho(\tau)\tau^{-1} d\tau=\infty$, when $\ep$ is small enough, there exist some positive constants $c_1$, $c_2$, and $c_3$ such that we have a unique solution satisfying
	\begin{equation*}
		\partial_t^{i} u\in C([0,T_\ep];H_D^{3-i}(M)) , \ 0\leq i\leq2,
	\end{equation*}
	with
	\begin{equation*}
		T_\ep=c_3\ep\left[\mathcal{H}_{2,3}^{inv}(c_1\ep^{-1}+\mathcal{H}_{2,3}(c_2\ep))\right]^{-1}.
	\end{equation*}
\end{theorem}

Theorem \ref{thm-3dim} is mainly established by local energy estimates and Sobolev embedding. Here, we give the definition of the compatibility of order $\kappa+1$.

\begin{definition}[The compatibility of order $\kappa+1$]
	For equation \eqref{eq-critnonlin-M}, note that there exist, formally, functions $\{\Psi_j\}_{0\leq j\leq \kappa+1}$ such that 
	\begin{equation*}
		\partial_t^ju(0,x)=\Psi_j(\nabla^{\leq j} f,\nabla^{\leq j-1}g)(x), \ x\in M , \  0\leq j\leq\kappa+1.
	\end{equation*}
	A priori, the initial data $(f,g)$ should fulfill
	\begin{equation*}
		\Psi_j(\nabla^{\leq j} f,\nabla^{\leq j-1}g)|_{\partial M}=0, \  0\leq j\leq \kappa.
	\end{equation*}
	Since the estimate of $\kappa$th-order energy needs $\|\partial_t^j u(0,\cdot)\|_{\dot{H}^1}<\infty$, $0\leq j \leq \kappa$ and $\partial_t^ju(0,x)\in L^2$, $1\leq j\leq\kappa+1$, by the compatibility of order $\kappa+1$, we mean that 
	\begin{gather*}
		f\in \dot{H}_D^1(M), \ \Psi_{j}(\nabla^{\leq j} f,\nabla^{\leq j-1}g)(x)\in H_D^1(M), \ 1\leq j\leq\kappa,\\
		\Psi_{\kappa+1}(\nabla^{\leq \kappa+1} f,\nabla^{\leq \kappa}g)(x)\in L^2(M).
	\end{gather*}
	
	Similarly, for Dirichlet-wave equations
	\begin{equation}\label{eq-Diriwave-M}
		\left\{\begin{aligned}
			&\Box u(t,x)=G(t,x)&&, \ (t,x)\in (0,T)\times M, \ \\
			&u(t,x)=0&&, \ x\in \partial M,  \ t>0 , \\
			&u(0,x)=f(x), \ u_t(0,x)=g(x) &&, \ x\in M, 
		\end{aligned}
		\right.
	\end{equation}
	there exist, formally, functions $\tilde{\Psi}_j$ such that
	\begin{equation*}
		\partial_t^ju(0,x)=\tilde{\Psi}_j(\nabla^{\leq j} f,\nabla^{\leq j-1}g, \partial^{\leq j-2}G), \ x\in M.
	\end{equation*}
	Then, the compatibility condition of order $k+1$ for the equation \eqref{eq-Diriwave-M} is
	\begin{equation*}
		\begin{gathered}
			f\in \dot{H}_D^1(M), \ \tilde{\Psi}_j(\nabla^{\leq j} f,\nabla^{\leq j-1}g, \partial^{\leq j-2}G) \in H_D^1(M), \  1\leq j \leq k,\\
			\tilde{\Psi}_{k+1}(\nabla^{\leq k+1} f,\nabla^{\leq k}g, \partial^{\leq k-1}G)(x)\in L^2(M).
		\end{gathered}
	\end{equation*}
\end{definition}

Using the argument in the proof of Theorem \ref{thm-3dim}, we can obtain the following result for $n\geq3$ and $\kappa>\max\{n/2,2\}$.

\begin{remark}\label{coro-Euclidgeq3}
	Let $n\geq3$, $\kappa > \max\{n/2,2\}$, and $\mathcal{K}$ be smooth, compact and nontrapping obstacles. There is a unique global solution to \eqref{eq-critnonlin-M} satisfying
	\begin{equation*}
		\partial_t^{i} u\in C([0,\infty);H_D^{\kappa+1-i}(M)) , \ 0\leq i\leq\kappa,
	\end{equation*}
	whenever the initial data satisfying the compatibility of order $\kappa+1$, and
	\begin{equation*}
		\sum_{|\alpha|\leq\kappa}\|(\nabla,\Omega)^\alpha(\nabla f,g)\|_{L^2(M)}=\ep\ll 1, \ \|f\|_{L^2(M)}<\infty.
	\end{equation*}
\end{remark}

Due to the lack of local energy estimates in dimension 2, the problem \eqref{eq-critnonlin-M} for non-empty obstacles is still unsettled. However, for $\mathcal{K}=\emptyset$, we will have the following theorem by standard energy estimates, Klainerman-Sobolev inequalities, and trace estimates.

\begin{theorem}\label{coro-Euclid2}
	Let $n=2$, $\mathcal{K}=\emptyset$, and $\kappa\geq 2$. If $\int_{0}^{1}\rho(\tau)\tau^{\kappa-4} d\tau<\infty$, there exists a small positive constant $\ep_2$, such that, the problem has a unique global solution to \eqref{eq-critnonlin-M} satisfying
\begin{equation*}
	\partial_t^{i} u\in C([0,\infty);H^{\kappa+1-i}(\R^2)) , \ 0\leq i\leq\kappa,
\end{equation*}
whenever the initial data satisfying
\begin{equation*}
	\sum_{|\alpha|\leq|\beta|\leq\kappa}\|x^\alpha\nabla^\beta f \|_{\dot{H}^1(\R^2)}+\|x^\alpha\nabla^\beta g \|_{L^2(\R^2)}=\ep\leq\ep_2, \ \|f\|_{L^2(\R^2)}<\infty.
\end{equation*}
Otherwise, if $\int_{0}^{1}\rho(\tau)\tau^{\kappa-4} d\tau=\infty$, there exist some positive $c_4$, $c_5$, and $c_6$ such that we have a unique solution satisfying
\begin{equation*}
	\partial^{i} u\in C([0,T_\ep];H^{\kappa+1-i}(\R^2)) , \ 0\leq i\leq\kappa,
\end{equation*}
with
\begin{equation}\label{eq-coro-Euclid2-lowerboundestimate}
	T_\ep=c_6\ep^{2}\left[\mathcal{H}_{\kappa,2}^{inv}(c_4\ep^{-2}+\mathcal{H}_{\kappa,2}(c_5\ep))\right]^{-2}.
\end{equation}
\end{theorem}

According to Theorem \ref{thm-3dim}, Remark \ref{coro-Euclidgeq3}, and Theorem \ref{coro-Euclid2}, as long as $\kappa>\frac{n}{2}$, the convergence of the integral 
\begin{equation}\label{eq-criterion}
	\int_{0}^{1}\frac{\rho(\tau)}{\tau}\tau^{\kappa-p_c(n)} d\tau
\end{equation}
foreshadows the existence of the global solution. And the divergence of the integral of \eqref{eq-criterion} implies a lower bound of the lifespan. For $n\geq3$, $\mathcal{K}=\overline{B_1}$, and $\kappa=1$, this criterion is still available for the radial solution and we should consider the following nonlinear problem
\begin{equation}\label{eq-critnonlin-radial}
	\left\{\begin{aligned}
		&\Box u(t,x)=F(\partial_t u,\partial_r u)&&, \ (t,x)\in (0,T)\times M, \ \\
		&u(t,x)=0&&, \ x\in \partial M,  \ t>0 , \\
		&u(0,x)=f(x), \ u_t(0,x)=g(x) &&, \ x\in M. 
	\end{aligned}
	\right.
\end{equation}
Here, $F\in C^\kappa(\{\mathbf{q}\in\R^2:|\mathbf{q}|\leq1\}),\  \kappa\geq1,  \ \partial_\mathbf{q}^{\leq\kappa} F(0)=0$.

\begin{theorem}\label{thm-radial}
	Let $n\geq3$, $\kappa=1$, $\mathcal{K}=\overline{B_1}$ and $(f,g)$ be radial functions. If $\int_{0}^{1}\rho(\tau)\tau^{-\frac{2}{n-1}-1} d\tau<\infty$, there exists a small positive constant $\ep_3$ such that the problem \eqref{eq-critnonlin-radial} has a unique radial global solution satisfying
	\begin{equation*}
		\partial_t^{i} u\in C([0,\infty);H_D^{2-i}(M)) , \ 0\leq i\leq1,
	\end{equation*}
	whenever the initial data satisfying the compatibility of order $2$, and
	\begin{equation*}
		\sum_{|\alpha|\leq1}\|\nabla^\alpha(\nabla f,g)\|_{L^2(M)}=\ep\leq\ep_3, \ \|f\|_{L^2(M)}<\infty.
	\end{equation*}
	Otherwise, if $\int_{0}^{1}\rho(\tau)\tau^{-\frac{2}{n-1}-1} d\tau=\infty$, there exist positive constants $c_7$, $c_8$, and $c_9$ such that we have a unique radial solution satisfying
	\begin{equation*}
		\partial_t^{i} u\in C([0,T_\ep];H_D^{2-i}(M)) , \ 0\leq i\leq1,
	\end{equation*}
	with
	\begin{equation}\label{eq-thm-radial-lowerboundestimate}
		T_\ep=c_9\ep^{\frac{2}{n-1}}\left[\mathcal{H}_{1,n}^{inv}(c_7\ep^{-\frac{2}{n-1}}+\mathcal{H}_{1,n}(c_8\ep))\right]^{-\frac{2}{n-1}}.
	\end{equation}
\end{theorem}

\begin{remark}\label{coro-radial-geq4}
	Let $n\geq4$, $\mathcal{K}=\overline{B_1}$, and $2\leq\kappa \leq n/2$. There is a unique radial global solution to \eqref{eq-critnonlin-radial} satisfying
	\begin{equation*}
		\partial_t^{i} u\in C([0,\infty);H_D^{\kappa+1-i}(M)) , \ 0\leq i\leq\kappa,
	\end{equation*}
	whenever the radial initial data satisfying the compatibility of order $\kappa+1$, and
	\begin{equation*}
		\sum_{|\alpha|\leq\kappa}\|\nabla^\alpha(\nabla f,g)\|_{L^2(M)}=\ep\ll 1, \ \|f\|_{L^2(M)}<\infty.
	\end{equation*}
\end{remark}

Following the argument in \cite{2306.11478}, we can show that the lower bound estimates of the lifespan are sharp for sample choices of the nonlinear terms and initial data. Before the statement, we give a rigorous definition of the lifespan.
\begin{definition}[Lifespan]
The lifespan of the equation \eqref{eq-critnonlin-M} is the supremum of the set consisting of $T>0$ such that there exists a unique solution 
\begin{equation*}
	\partial_t^{i} u\in C([0,T];H_D^{2-i}(M)) , \ 0\leq i\leq1.
\end{equation*}

Assume $\ep$ is sufficiently small. For fixed $\phi, \ \psi\in C_c^\infty(M)$, let $T(\ep)$ denote the lifespan of the equations \eqref{eq-critnonlin-M} with initial data $f=\ep\phi, \ g=\ep\psi$.
\end{definition}

\begin{theorem}[Sharpness of the lifespan estimates]\label{thm-lifespan-sharp}
	Let $n\geq3$, $a\in(0,\frac{2}{n-1}]$, and $M=\{x:|x|>1\}$. 
	Consider the following equation,
	\begin{equation}\label{eq-blowup-upperbound-n}
		\left\{\begin{aligned}
			&\Box u(t,x)=|\partial_t u|^{1+a}\mu(|\partial_t u|)(t,x)\coloneq F_a(\partial_tu)&&, \ (t,x)\in (0,T)\times M, \ \\
			&u(t,x)=0&&, \ x\in \partial M,  \ t>0 , \\
			&u(0,x)=\ep\phi(x), \ \partial_tu(0,x)=0 &&, \ x\in M.
		\end{aligned}
		\right.
	\end{equation}
	Here, the initial datum $\phi\in C_c^\infty(M)$ is non-negative and radial with $\supp \phi=\{x:2\leq|x|\leq3\}$. The function $\mu \in C([0,1])\cap C^1((0,1])$ such that
	\begin{equation}\label{eq-prop-upperboundestimate-mucondition}
		\mu(0)=0, \  0\leq\tau\mu(\tau)\lesssim\mu(\tau), \ \int_{0}^{1}\mu(\tau)\tau^{a-p_c(n)}=\infty,
	\end{equation}
	and $F_a(\tau)$ is a convex function on $[-1,1]$. There exist positive constants $k_1$, $k_2$, and $k_3$ such that, when $\ep$ is sufficiently small, the lifespan $T(\ep)$ of the solution to \eqref{eq-blowup-upperbound-n} has the upper bound,
	\begin{equation}\label{eq-blowup-prop-upperbound-sharp}
		T(\ep)\leq k_3\ep^{\frac{2}{n-1}}[\mathcal{H}_{1,n}^{inv}(k_1\ep^{-\frac{2}{n-1}}+\mathcal{H}_{1,n}(k_2\ep))]^{-\frac{2}{n-1}}.
	\end{equation}
	Furthermore, estimates \eqref{eq-blowup-prop-upperbound-sharp} with $\mathcal{H}_{1,n}$ replaced by $\mathcal{H}_{2,2}$ is still hold when $n=2$, $a\in [1,2]$, $M=\R^2$ and $\mu\in C([0,1])\cap C^2((0,1])$ satisfying
	\begin{equation}\label{eq-prop-upperboundestimate-mucondition-dim2}
		\mu(0)=0, \ \mu'(\tau)\geq0, \  \int_{0}^{1}\mu(\tau)\tau^{a-3}=\infty, \ \text{and} \  |\tau^i\mu^{(i)}(\tau)|\lesssim\mu(\tau), \ 1\leq i \leq 2.
	\end{equation}
\end{theorem}
\begin{remark}
	The requirements \eqref{eq-prop-upperboundestimate-mucondition} and \eqref{eq-prop-upperboundestimate-mucondition-dim2} for $\mu$ are to insure the local existence of the solution $\partial_t^{i} u\in CH_D^{2-i} , \ 0\leq i\leq1$.
\end{remark}

To conclude this section, we give some typical examples of $F(\partial u)$.
\begin{itemize}
	\item Example 1: $F(\partial u)=|\partial u|^p$, $p>1$ for $n\geq3$ and $p>2$ for $n\geq2$. We can choose
	\begin{equation*}
		\kappa=\left\{\begin{aligned}
			&\text{the integer part of} \ p, \ p \ \text{is not an integer}\\
			&p-1, \ p \ \text{is an integer}.
		\end{aligned}\right.
	\end{equation*}
	Then, the criterion is whether $\int_{0}^{1}\tau^{p-p_c(n)-1}$ is infinity or not, which coincides with the classical results. For $p<p_c(n)$, $H_{\kappa,n}(s)\sim s^{p-p_c(n)}-1$, hence, for $\ep\ll1$,
	\begin{equation*}
		\ep^{\frac{2}{n-1}}\left[\mathcal{H}_{\kappa,n}^{inv}(K_9\ep^{-\frac{2}{n-1}}+\mathcal{H}_{\kappa,n}(K_{10}\ep))\right]^{-\frac{2}{n-1}} \sim \ep^{\frac{2(p-1)}{2-(n-1)(p-1)}},
	\end{equation*}
	where $K_9,K_{10}$ are positive constants given in the lifespan estimates above;
	for $p=p_c(n)$,
	\begin{equation*}
		H_{\kappa,n}(s)\sim \ln(s^{-1}),
	\end{equation*}
	and then, for $\ep\ll1$
	\begin{equation*}
		\ln\left(\ep^{\frac{2}{n-1}}\left[\mathcal{H}_{\kappa,n}^{inv}(K_9\ep^{-\frac{2}{n-1}}+\mathcal{H}_{\kappa,n}(K_{10}\ep))\right]^{-\frac{2}{n-1}}\right)\sim \ep^{-\frac{2}{n-1}}.
	\end{equation*}
	\item Example 2: $F(\partial u)=|\partial u|^{p_c(n)}\mu(|\partial u|)$ with $\mu$ satisfying \eqref{eq-prop-upperboundestimate-mucondition} or \eqref{eq-prop-upperboundestimate-mucondition-dim2}.
	We can select $\kappa$ to be the integer part of $p_c(n)$. Then, the criterion is whether $\int_{0}^{1}\mu(\tau)\tau^{-1}$ is infinity or not, which coincides with the blow-up results in \cite{2306.11478}. By direct calculations,
	\begin{equation*}
		H_{\kappa,n}(s)\sim \int_{s}^{1}\mu(\tau)\tau^{-1}d\tau,
	\end{equation*}
	hence we extend the existence results in \cite{2306.11478}.
	\item Example 3: $F(\partial u)=|\partial u|\frac{1}{\log(|\partial u|^{-1})}$ with $n\geq3$ and $M=\{x:|x|>1\}$. By Theorem \ref{thm-radial}, given radial initial data, we can obtain a unique radial solution locally and provide the lower bound estimate of the lifespan. However, the sharp upper bound estimate is still open.
\end{itemize}

\section{Preliminary}
\subsection{Notations}
In this subsection, we list the notations used in this paper.

\begin{itemize}
	\item Besides $(t,x)=(x_0, x_1, x_2, \dots, x_n)\in\R^{1+n}$, we will use polar coordinates $(t,r,\omega)\in \R\times[0,\infty)\times\mathbb{S}^{n-1}$.
	\item Denote $\partial_i=\partial/\partial x_i$, $0\leq i\leq n$, with the abbreviation $\partial=(\partial_0, \partial_1, \dots, \partial_n)=(\partial_t, \nabla)$. We will also need $\partial_r=\frac{x}{r}\cdot\nabla$ and
	\begin{gather*}
		L_j=t\partial_j+x_j\partial_0 \  1\leq j \leq n, \\
		\Omega_{ij}=x_i\partial_j-x_j\partial_i \ 1\leq i<jk\leq n, \\
		S=t\partial_0+x\cdot\nabla. 
	\end{gather*}
	We denote
	\begin{gather*}
			L=(L_1,\dots,L_n), \ \Omega=(\Omega_1,\dots,\Omega_{n(n-1)/2}),\\
			\Gamma=(\partial,\Omega,L,S) , \ Y=(\nabla,\Omega),  \ Z=(\partial,\Omega).
	\end{gather*}
	Hereby in this notation list, $X=(X_1,X_2,\dots,X_\nu)$ will be one of the collections of vector fields, $\partial_t$, $\nabla$, $\partial$, $\Gamma$, $\Omega$, and $Z$.
	\item For multi-indices, $\alpha\in\N^{\nu}$, $X^\alpha=X_1^{\alpha_1}X_2^{\alpha_2}\cdots X_\nu^{\alpha_\nu}$ and $|\alpha|=\sum_1^\nu \alpha_i$. For any $k\in\N$, $X^{\leq k}$ and $X^{k}$ mean the collection of the vector fields $(X^\alpha)_{|\alpha|\leq k}$ and $(X^\alpha)_{|\alpha|= k}$. 
	\item With the Dirichlet boundary condition, we define $\dot{H}_D^1(M)$ as the closure of $f\in C_c^\infty(M)$, with respect to the norm
	\begin{equation*}
		\|f\|_{\dot{H}_D^1(M)}=\|\nabla f\|_{L^2(M)}.
	\end{equation*}
	If $M=\R^n$, $\dot{H}_D^1=\dot{H}^1$ is the closure of $C_c^\infty(\R^n)$ with respect to the $\dot{H}^1$ norm. 
	Also, we denote $H_D^i(M)=\dot{H}_D^1(M)\cap H^i(M)$ for $i\geq1$ and $H_D^0(M)=L^2(M)$. 
	Space $H_0^1(M)$ is the closure of $C_c^\infty(M)$ with respect to the $H^1$ norm.
	Space $(H^{-1}(M))_{w*}$ is the dual space of $H_0^1(M)$ equipped with weak* topology  .
	\item A function $u(t,x)$ belonging to $C_t(L_{x,loc}^2(M))$ means that for any $\sigma\in C_c^\infty(\R^n)$, $\sigma u\in C_t(L_{x}^2(M))$. We denote $C_t^i((H^{-1}(M))_{w*})$ the $i$-times continuous differentiable $(H^{-1}(M))_{w*}$-valued function space. 
	\item For a Banach space $F$, $C_b([0,\infty);F)\coloneq C\cap L^\infty([0,\infty);F)$.
	\item The space $l^s_q(F)(1\leq q\leq\infty)$ means 
	\begin{equation*}
		\|u\|_{l^s_q(F)}=\|(\|2^{js}\Phi_j(x)u(t,x)\|_F)\|_{l_{j\geq0}^q},
	\end{equation*}
	for a partition of unity subordinate to the (inhomogeneous) dyadic (spatial) annuli, $\sum_{j\geq0}\Phi_j(x)=1$. A typical choice could be a radial, non-negative $\Phi_0(x)\in C_0^\infty$ with value 1 for $|x|\leq1$ and 0 for $|x|\geq2$, and $\Phi_j(x)=\Phi_0(2^{-j}x)-\Phi_0(2^{1-j}x)$ for $j\geq1$. Take $\Phi_{-1}=0$ and then $\Phi_j(\Phi_{j-1}+\Phi_j+\Phi_{j+1})=\Phi_j$, $j\geq0$.
	\item We will use the norm in polar coordinates $(r,\omega)$,
	\begin{equation*}
		\|f\|_{L_t^{q_1}L_r^{q_2}L_\omega^{q_3}}=\left\|\left(\int_{0}^{\infty}\|f(t,r\omega)\|_{L_\omega^{q_3}}^{q_2}r^{n-1}dr\right)^{\frac{1}{q_2}}\right\|_{L^{q_1}(t>0)},
	\end{equation*}
	with trivial modification for the case $q_2=\infty$, where $L_\omega^{q_3}$ is the standard Lebesgue space on the sphere $\mathbb{S}^{n-1}$.
	\item The energy norm is
	\begin{equation*}
		\|u\|_E=\|u\|_{E_0}=\|\partial u\|_{L_t^\infty L_x^2((0,T)\times M)}.
	\end{equation*}
	For a collection of vector fields $X$, the energy norm of order $m$ is
	\begin{equation*}
		\|u\|_{E_{m,X}}=\sum_{|\alpha|\leq m}\|X^\alpha u\|_E\eqcolon\|X^{\leq m} u\|_E.
	\end{equation*}
	Also, we use $\|\cdot\|_{LE}$ to denote the local energy norm
	\begin{equation*}
			\|u\|_{LE}=\|\partial u\|_{l_\infty^{-\half}L_t^2 L_x^2((0,T)\times M)}+\|u/r\|_{l_\infty^{-\half}L_t^2 L_x^2((0,T)\times M)}.
	\end{equation*}
	and the dual norm $LE^*=l_1^{\half}L_t^2 L_x^2((0,T)\times M)$.
	Similarly, for a collection of vector fields $X$, the local energy norm of order $m$ is
	\begin{equation*}
		\|u\|_{LE_{m,X}}=\sum_{|\alpha|\leq m}\|X^\alpha u\|_{LE}\eqcolon\|X^{\leq m} u\|_{LE}.
	\end{equation*} 
	\item $\|u\|_{F_1+F_2}=\inf_{u=u_1+u_2}(\|u_1\|_{F_1}+\|u_2\|_{F_2})$, $\|u\|_{F_1\cap F_2}=\max\{\|u\|_{F_1},\|u\|_{F_2}\}$.
	\item For a subset $L\subset\R^n$, $\chi_L$ is the characteristic function of $L$ and $\bar{\chi}_L=1-\chi_L$. For $s\geq0$, denote $B_s\subset\R^n$ the open ball centered at the origin with a radius $s$, $\chi_s$ the characteristic function of the set $B_{\<s\>}$ and $\bar{\chi}_s=1-\chi_s$. 
	\item We will use $A\lesssim B$ to stand for $A\leq CB$ where the constant $C$ may change line to line. Then, $A\sim B$ stands for $A\lesssim B\lesssim A$. In addition, when denoting by $(a)_+$ for $a\in\R$, we mean the relevant estimate holds for $a+\delta$ for sufficiently small $\delta>0$.
\end{itemize}

\subsection{Estimates}
In this subsection, we give some estimates.

\begin{proposition}[\rm{Klainerman-Sobolev Inequality}]\label{prop-KSinequality}
	If $1\leq p <\infty$ and $s>n/p$, then the inequality
	\begin{equation*}
		(1+|t|+|x|)^\frac{n-1}{p}(1+||t|-|x||)^\frac{1}{p}|v(t,x)|\leq C_{KS}\|\Gamma^{\leq s}v(t,\cdot)\|_{L_x^p}
	\end{equation*}
	holds. If $1\leq p<q<\infty$ and $\frac{1}{q}\geq\frac{1}{p}-\frac{s}{n}$, then 
	\begin{equation*}
		\|\chi_tv(t,\cdot)\|_{L_x^q}\lesssim(1+|t|)^{-n(\frac{1}{p}-\frac{1}{q})}\|\Gamma^{\leq s}v(t,\cdot)\|_{L_x^p}.
	\end{equation*}
\end{proposition}

See \cite[Chapter \uppercase\expandafter{\romannumeral2}]{MR2455195} and \cite[Theorem 3.4.1, Theorem 3.4.2]{MR3729480} for the proof.

We will need the variant of the Sobolev embedding.

\begin{lemma}[\rm{Sobolev embedding}]\label{lem-Sobolev-rotaition}
	Let $n\geq2$. For $k\geq\frac{n}{2}-\frac{n}{q}$ with $q\in[2,\infty)$, we have
	\begin{equation*}
		\|\<r\>^{(n-1)(\half-\frac{1}{q})}u\|_{L^q(M)}\lesssim\sum_{|\alpha|\leq k}\|Y^\alpha u\|_{L^2(M)}.
	\end{equation*}
	Moreover, we have
	\begin{equation*}
		\|\<r\>^{\frac{n-1}{2}}u\|_{L^\infty(M)}\leq C_S\sum_{|\alpha|\leq \frac{n+2}{2}}\|Y^\alpha u\|_{L^2(M)}.
	\end{equation*}
\end{lemma}

See \cite[Lemma 2.2]{MR3338309} for the proof.

We need the following trace estimates.
\begin{proposition}\label{prop-traceestimate-Rn}
	If $p\in[2,4]$ and $n=2$, we have
	\begin{equation*}
		\sup_{r>0} r^\frac{1}{2}\|v(r\cdot)\|_{L_\omega^p}\lesssim\|\partial_rv\|_{L^2(\R^2)}^\half\|\Omega^{\leq1} v\|_{L^2(\R^2)}^\half.
	\end{equation*}
\end{proposition}
See \cite[Proposition 2.4]{MR3552253} for the proof.

\begin{proposition}\label{prop-traceestimate-M}
	Let $n\geq2$, then
	\begin{equation*}
		\|r^{\frac{n-1}{2}}u(r\omega)\|_{L_\omega^2}\lesssim \|u\|_{L^2(|x|\geq r)}+\|\nabla u\|_{L^2(|x|\geq r)}.
	\end{equation*}
\end{proposition}
See \cite[Lemma 2.2]{MR2980460} for the proof.

\begin{proposition}[\rm{Local energy estimates}]\label{prop-localenergy}
	Consider the Dirichlet-wave equations \eqref{eq-Diriwave-M}.
	Let $n\geq3$. For $(f,g,G)\in \dot{H}_D^1\times L_x^2\times(LE^*+L_t^1L_x^2)$, we have $u\in C_b([0,\infty);\dot{H}_D^1(M)$, $\partial_t u\in C_b([0,\infty);L_x^2(M))$, and
	\begin{equation*}
		\|u\|_{LE\cap E}\lesssim\|f\|_{\dot{H}_D^1(M)}+\|g\|_{L_x^2}+\|G\|_{LE^*+L_t^1L_x^2}.
	\end{equation*}
	In addition, we have higher order ($k$th-order) local energy estimates. 
	Provided that $(f,g)\in H_D^{k+1}\times H_D^k$, that $G\in C_t^{k-1}((H^{-1}(M))_{w*})$ with $Z^{\leq k}G\in LE^*+L_t^1L_x^2$ and $\partial^{\leq k-1}G\in C_b(L_{x,loc}^2(M))$, and that $(f,g,G)$ fulfill the compatibility condition of order $k+1$, there exists some $R>4$ such that we have a solution $u\in C([0,\infty);H_D^{k+1}(M))$ and
	\begin{gather*}
		\partial_t^i u\in C_b([0,\infty);H_D^{k+1-i}(M)), \ 1\leq i \leq k+1, \\
		\partial_t^jY^\alpha u\in C_b([0,\infty);H^{k+1-j-|\alpha|}(M)),  \ 1\leq|\alpha|\leq\min \{k+1-j,k\} ,
	\end{gather*}
	with the estimates
	\begin{equation}\label{eq-prop-localenergy-highorder}
		\begin{aligned}
			\|u\|_{LE_{k,Z}\cap E_{k,Z}}&\lesssim\sum_{|\alpha|\leq k}\|(\nabla,\Omega)^\alpha(\nabla f, g)\|_{L_x^2}+\|Z^\alpha G\|_{LE^*+L_t^1L_x^2}\\
			&\quad+\sum_{|\gamma|\leq k-1}\|Z^\gamma G(0,x)\|_{L_x^2}+\|\partial^\gamma G\|_{(L_t^\infty\cap L_t^2)L_x^2(B_{2R})}.
		\end{aligned}
	\end{equation}
\end{proposition}
	See \cite[Proposition 1.4]{MR3338309} for the proof of estimate and Appendix \ref{section-regularity} for the explanation of regularity.
	
\begin{lemma}\label{lem-seriessum}
	Let $b\in\R$ and $s>1$. For any non-decreasing function $h$ defined on $[0,1]$ with $h(0)\geq0$, we have the following inequality
		\begin{equation*}
		\sum_{j=0}^{2^j<s} h(2^{-j\frac{n-1}{2}})2^{-j(b+1)}\leq h(1)+\frac{2(b+1)}{(n-1)(2^{b+1}-1)}\int_{s^{-\frac{n-1}{2}}}^{1} \frac{h(\tau)}{\tau}\tau^{(b+1)\frac{2}{n-1}}d\tau.
	\end{equation*}
\end{lemma}
\begin{proof}
	Note that
	\begin{equation*}
		\int_{2^{-j}}^{2^{-j+1}}\tau^b d\tau=\frac{2^{b+1}-1}{b+1}2^{-j(b+1)}.
	\end{equation*}
	Thus, 
	\begin{align*}
		\sum_{j=1}^{2^j<s} h(2^{-j\frac{n-1}{2}})2^{-j(b+1)}
		&\leq \sum_{j=1}^{2^j<s} \frac{b+1}{2^{b+1}-1} h(2^{-j\frac{n-1}{2}})\int_{2^{-j}}^{2^{-j+1}}\tau^b d\tau \\
		&\leq  \frac{b+1}{2^{b+1}-1} \int_{s^{-1}}^{1}h(\tau^\frac{n-1}{2})\tau^b d\tau \\
		&=\frac{2(b+1)}{(n-1)(2^{b+1}-1)}\int_{s^{-\frac{n-1}{2}}}^{1} \frac{h(\tau)}{\tau}\tau^{(b+1)\frac{2}{n-1}}d\tau.\qedhere
	\end{align*}
\end{proof}

\subsection{Solution spaces}
The solution will be obtain by iteration in some suitable function space.

Let $(X,\tilde{X})$ be $(\partial,\nabla)$ or $(Z,Y)$. For initial data $(f,g)\in H_D^{\kappa+1}\times H_D^{\kappa}$ satisfying the compatibility of order $\kappa+1$, denote $S_{\kappa,X,T}(T\in(0,\infty))$ the space consisting of functions
\begin{equation*}
	\begin{gathered}
		\partial_t^i u\in C([0,T];H_D^{\kappa+1-i}(M)), \ 0\leq i \leq \kappa+1,\\
		\partial_t^i \tilde{X}^\alpha u\in C([0,T];H^{\kappa+1-i-|\alpha|}(M)), \ 1\leq|\alpha|\leq \min\{\kappa+1-i,\kappa\},\\
		\|X^{\leq\kappa} u\|_{LE}<\infty, \\
		\partial_t^iu(0,x)=\Psi_i(\nabla^{\leq i} f,\nabla^{\leq i-1}g)(x) , \  0\leq i\leq\kappa+1,
	\end{gathered}
\end{equation*}
and equipped with the norm $\|\cdot\|_{LE_{\kappa,X}\cap E_{\kappa,X}}$. Also, we define $S_{\kappa,X,\infty}$ as the space consisting of functions $u\in C([0,\infty);H_D^{\kappa+1}(M))$ and
\begin{gather*}
	\partial_t^i u\in C_b([0,\infty);H_D^{\kappa+1-i}(M)), \ 1\leq i \leq \kappa+1, \\
	\partial_t^i\tilde{X}^\alpha u\in C_b([0,\infty);H^{\kappa+1-i-|\alpha|}(M)),  \ 1\leq|\alpha|\leq \min\{\kappa+1-i,\kappa\} ,\\
	\|Z^{\leq\kappa} u\|_{LE}<\infty, \\
	\partial_t^iu(0,x)=\Psi_i(\nabla^{\leq i} f,\nabla^{\leq i-1}g)(x) , \  0\leq i\leq\kappa+1,
\end{gather*}
and equipped with the norm $\|\cdot\|_{LE_{\kappa,X}\cap E_{\kappa,X}}$. 
Denote 
\begin{equation*}
	S_{\kappa,X,T}(\delta)=\{u\in S_{\kappa,X,T}:\|u\|_{LE_{\kappa,X}\cap E_{\kappa,X}}\leq \delta\} ,  \ T\in(0,\infty], \ \delta>0.
\end{equation*}

We choose $X$ to be $Z$ in the proofs of Theorem \ref{thm-3dim} and Remark \ref{coro-Euclidgeq3} and to be $\partial$ in the proofs of Theorem \ref{thm-radial} and Remark \ref{coro-radial-geq4}.

For Theorem \ref{coro-Euclid2}, we will use $S_{\kappa,\Gamma,T}(T\in(0,\infty))$ the space consisting of function
\begin{equation*}
	\begin{gathered}
		u\in C([0,T];H^{\kappa+1}(\R^2)), \\
		\partial\Gamma^\alpha u\in C([0,T];H^{\kappa-|\alpha|}(\R^2)), \ 0\leq |\alpha| \leq \kappa,\\
		\partial_t^iu(0,x)=\Psi_i(\nabla^{\leq i} f,\nabla^{\leq i-1}g)(x) , \  0\leq i\leq\kappa+1,
	\end{gathered}
\end{equation*}
and equipped with the norm $\|\cdot\|_{E_{\kappa,\Gamma}}$. Also, we define $S_{\kappa,\Gamma,\infty}$ as the space consisting of functions $u\in C([0,\infty);H^{\kappa+1}(\R^2))$ and
\begin{equation*}
	\begin{gathered}
		\partial\Gamma^\alpha u\in C_b([0,\infty);H^{\kappa-|\alpha|}(\R^2)), \ 0\leq |\alpha| \leq \kappa,\\
		\partial_t^iu(0,x)=\Psi_i(\nabla^{\leq i} f,\nabla^{\leq i-1}g)(x) , \  0\leq i\leq\kappa+1,
	\end{gathered}
\end{equation*}
and equipped with the norm $\|\cdot\|_{E_{\kappa,\Gamma}}$. Again, denote 
\begin{equation*}
	S_{\kappa,\Gamma,T}(\delta)=\{u\in S_{\kappa,\Gamma,T}:\|u\|_{E_{\kappa,\Gamma}}\leq \delta\} ,  \ T\in(0,\infty], \ \delta>0.
\end{equation*}
	
\section{Proof of Theorem \ref{thm-3dim}}
In this section, $E_i$, $LE_i$, and $S_T$ stand for $E_{i,Z}$, $LE_{i,Z}$, and $S_{2,Z,T}$.
Theorem \ref{thm-3dim} will be established by iteration argument in some suitable function space $S_T$.
\begin{proposition}\label{prop-proof-3dim-iteration}
	Let $T\in(0,\infty]$ and $(f,g,F)$ satisfy the conditions in Theorem \ref{thm-3dim}. When $T=\infty$, further require that $\int_{0}^{1}\frac{\rho(\tau)}{\tau}d\tau<\infty$. Temporarily assume that $S_T(1/(2C_S))$ are non-empty. For all $\tilde{u}\in S_T(1/(2C_S))$, let $I[\tilde{u}]$ be the solution to the problem
	\begin{equation}\label{eq-proof-thm-3dim-prop-iteration}
		\left\{\begin{aligned}
			&\Box u(t,x)=F(\partial \tilde{u})&&, \ (t,x)\in (0,T)\times M, \ \\
			&u(t,x)=0&&, \ x\in \partial M,  \ t>0 , \\
			&u(0,x)=f(x), \ u_t(0,x)=g(x) &&, \ x\in M,
		\end{aligned}
		\right.
	\end{equation}
and fulfill the estimates \eqref{eq-prop-localenergy-highorder}. Then, there exists constants $C_1$ and $C_2$, such that, $\tilde{u},\tilde{v}\in S_T(1/(2C_S))$, implies
	\begin{gather}
			\|I[\tilde{u}]\|_{LE_2\cap E_2}\leq C_1\ep+C_1\Lambda(T,\|\tilde{u}\|_{E_2})\|\tilde{u}\|_{E_2}\|\tilde{u}\|_{LE_2\cap E_2} , \label{eq-prop-3dim-2regular-fbound} \\
			\|I[\tilde{u}]-I[\tilde{v}]\|_{LE_{1}\cap E_1}\leq C_2\Lambda(T,P(\tilde{u},\tilde{v}))P(\tilde{u},\tilde{v})\|\tilde{u}-\tilde{v}\|_{LE_1\cap E_1},\label{eq-prop-3dim-2regular-fdiffer}
	\end{gather}
	where
	\begin{equation*}
		\Lambda(T,\zeta)=\rho(C_S\zeta)+\int_{{\<T\>}^{-1}}^{1}\frac{\rho(C_S \zeta\tau)}{\tau}d\tau ,
	\end{equation*}
	and $P(\tilde{u},\tilde{v})=\max
	\{\|\tilde{u}\|_{LE_2\cap E_2},\|\tilde{v}\|_{LE_2\cap E_2}\}$.
\end{proposition}

\begin{proof}[Proof of Proposition \ref{prop-proof-3dim-iteration}]
	According to Lemma \ref{lem-Sobolev-rotaition}, we have $$\|\partial \tilde{u}\|_{\infty}\leq C_S{\<r\>}^{-1}\|\tilde{u}\|_{E_2}<\half.$$ 
	Because $\partial \tilde{u}\in C_tH^{2}(M)\cap C_t^2L_x^2(M)$, it follows that, by the chain rule, $F(\partial \tilde{u})\in C_t^2((H^{-1}(M))_{w*})$, $\partial^{\leq 1}[F(\partial \tilde{u})]\in C_bL_{x,loc}^2$, and
	\begin{equation*}
		\begin{gathered}
			|F(\partial \tilde{u})|\lesssim \rho(C_S \|\tilde{u}\|_{E_2} {\<r\>}^{-1})|\partial \tilde{u}|^2, \\
			|Z[F(\partial \tilde{u})]|\lesssim \rho(C_S \|\tilde{u}\|_{E_2} {\<r\>}^{-1})|\partial \tilde{u}||Z\partial \tilde{u}|,\\
			|Z^{2}[F(\partial \tilde{u})]|\lesssim \rho(C_S \|\tilde{u}\|_{E_2} {\<r\>}^{-1})\left(|\partial \tilde{u}||Z^{2}\partial \tilde{u}|+|Z\partial \tilde{u}|^2\right).
		\end{gathered}
	\end{equation*}
	Also, the compatibility condition of the problem \eqref{eq-proof-thm-3dim-prop-iteration} coincides with that of the problem \eqref{eq-critnonlin-M}.
	Thus, by Proposition \ref{prop-localenergy}, as long as $Z^{\leq2} [F(\partial \tilde{u})]\in LE^*+L_t^1L_x^2$, it follows 
	\begin{equation}\label{eq-proof-thm-3dim-2orderenergyiteration}
		\begin{aligned}
			\|I[\tilde{u}]\|_{LE_2\cap E_2}&\lesssim\sum_{|\alpha|\leq 2}\|(\nabla,\Omega)^\alpha(\nabla f, g)\|_{L_x^2}+\|Z^\alpha [F(\partial \tilde{u})]\|_{LE^*+L_t^1L_x^2}\\
			&\quad+\sum_{|\gamma|\leq 1}\|Z^\gamma [F(\partial \tilde{u})](0,x)\|_{L_x^2}+\|\partial^\gamma [F(\partial \tilde{u})]\|_{(L_t^\infty\cap L_t^2)L_x^2(B_{2R})}\\
			&\lesssim \ep+\rho(C_S \|\tilde{u}\|_{E_2})\|\tilde{u}\|_{E_2}\|\tilde{u}\|_{LE_1\cap E_1} +\sum_{|\alpha|\leq 2}\|Z^\alpha [F(\partial \tilde{u})]\|_{LE^*+L_t^1L_x^2}.
		\end{aligned}
	\end{equation}
	According to the definition of $\|\cdot\|_{LE^*+L_t^1L_x^2}$, we have
	\begin{equation}\label{eq-proof-thm-3dim-nonlinerestimate}
		\|Z^\alpha [F(\partial \tilde{u})]\|_{LE^*+L_t^1L_x^2}\leq \|\chi_TZ^\alpha [F(\partial \tilde{u})]\|_{LE^*}+\|\bar{\chi}_TZ^\alpha [F(\partial \tilde{u})]\|_{L_t^1L_x^2}.
	\end{equation}
	Noting that $\supp \Phi_j\subset\{x:2^{j-1}\leq|x|\leq2^{j+1}\}$, $j\geq1$, we obtain that
	\begin{equation*}
		\begin{aligned}
			&\quad \|\chi_TZ^{\leq 2} [F(\partial \tilde{u})]\|_{LE^*}\\
			&\lesssim \sum_{j=0}^{2^{j-1}<\<T\>}2^{j\half}\|\Phi_j(x) \rho(C_S\|\tilde{u}\|_{E_2}{\<r\>}^{-1})\left(|\partial \tilde{u}||Z^{\leq2}\partial \tilde{u}|+|Z\partial \tilde{u}|^2\right)\|_{L_t^2L_x^2}. \\
		\end{aligned}
	\end{equation*} 
	By Lemma \ref{lem-Sobolev-rotaition} and \ref{lem-seriessum}, we deduce that
	\begin{align*}
			&\quad\sum_{j=1}^{2^{j-1}<\<T\>}2^{j\half}\|\Phi_j(x)\rho(C_S \|\tilde{u}\|_{E_2}{\<r\>}^{-1}) |\partial \tilde{u}||Z^{\leq2}\partial \tilde{u}|\|_{L_t^2L_x^2}\\
			&\lesssim \sum_{j=1}^{2^{j-1}<\<T\>}2^{j\half}\rho(C_S \|\tilde{u}\|_{E_2}2^{-(j-1)})2^{-(j-1)}\|\Phi_j(x) |Z^{\leq2}\partial \tilde{u}|\|_{L_t^2L_x^2}\|\tilde{u}\|_{E_2}\\
			&\lesssim{\left(\sum_{j=1}^{2^{j-1}<\<T\>}\rho(C_S \|\tilde{u}\|_{E_2}2^{-(j-1)})\right)} \|\tilde{u}\|_{E_2}\|\tilde{u}\|_{LE_2}\\
			&\lesssim \left(\rho(C_S\|\tilde{u}\|_{E_2})+\int_{{\<T\>}^{-1}}^{1} \frac{\rho(C_S\|\tilde{u}\|_{E_2}\tau)}{\tau}d\tau\right)\|\tilde{u}\|_{E_2}\|\tilde{u}\|_{LE_2},
	\end{align*}
	and
	\begin{align*}
			&\quad\sum_{j=1}^{2^{j-1}<\<T\>}2^{j\half}\|\Phi_j(x)\rho(C_S \|\tilde{u}\|_{E_2}{\<r\>}^{-1}) |Z\partial \tilde{u}|^2\|_{L_t^2L_x^2}\\
			&\lesssim \sum_{j=1}^{2^{j-1}<\<T\>}2^{j\half}\rho(C_S \|\tilde{u}\|_{E_2}2^{-(j-1)})2^{-(j-1)}\|\|{\<r\>}^\half Z\partial \tilde{u}\|_{L_x^4}\cdot\|{\<r\>}^\half\Phi_j(x)Z\partial \tilde{u}\|_{L_x^4}\|_{L_t^2}\\
			&\lesssim \sum_{j=1}^{2^{j-1}<\<T\>}2^{-j\half}\rho(C_S \|\tilde{u}\|_{E_2}2^{-(j-1)})\|\tilde{u}\|_{E_2}\|\Phi_j(x)Y^{\leq1}Z\partial \tilde{u}+[\partial^{\leq1}\Phi_j(x)]Z\partial \tilde{u}\|_{L_t^2L_x^2}\\
			&\lesssim \left(\rho(C_S\|\tilde{u}\|_{E_2})+\int_{{\<T\>}^{-1}}^{1} \frac{\rho(C_S\|\tilde{u}\|_{E_2}\tau)}{\tau}d\tau\right)\|\tilde{u}\|_{E_2}\|\tilde{u}\|_{LE_2}\\
			&\quad+\sum_{j=1}^{2^{j-1}<\<T\>}\sum_{i=-1}^{1}2^{-j\half}\rho(C_S \|\tilde{u}\|_{E_2}2^{-(j-1)})\|\tilde{u}\|_{E_2}\|\Phi_{j+i}(x)Z\partial \tilde{u}\|_{L_t^2L_x^2}\\
			&\lesssim\left(\rho(C_S\|\tilde{u}\|_{E_2})+\int_{{\<T\>}^{-1}}^{1} \frac{\rho(C_S\|\tilde{u}\|_{E_2}\tau)}{\tau}d\tau\right)\|\tilde{u}\|_{E_2}\|\tilde{u}\|_{LE_2}.
	\end{align*}
	Therefore, it follows
	\begin{equation}\label{eq-proof-thm-3dim-nonlinerestimate-smaller}
		\begin{aligned}
			&\quad\|\chi_TZ^{\leq 2} [F(\partial \tilde{u})]\|_{LE^*}\\
			&\lesssim{\left(\rho(C_S\|\tilde{u}\|_{E_2})+\int_{{\<T\>}^{-1}}^{1} \frac{\rho(C_S\|\tilde{u}\|_{E_2}\tau)}{\tau}d\tau\right)}\|\tilde{u}\|_{E_2}\|\tilde{u}\|_{LE_2},
		\end{aligned}
	\end{equation}
	and, if $\int_{0}^{1}\frac{\rho(\tau)}{\tau}d\tau<\infty$,
	\begin{equation}\label{eq-proof-thm-3dim-nonlinerestimate-smaller-Tinfty}
		\|Z^{\leq 2} [F(\partial \tilde{u})]\|_{LE^*}\lesssim{\left(\rho(C_S\|\tilde{u}\|_{E_2})+\int_{0}^{1} \frac{\rho(C_S\|\tilde{u}\|_{E_2}\tau)}{\tau}d\tau\right)}\|\tilde{u}\|_{E_2}\|\tilde{u}\|_{LE_2}.
	\end{equation}
	If $T$ is finite, due to $\tilde{u}\in S_T(1/(2C_S))$, we have 
	\begin{equation}\label{eq-proof-thm-3dim-nonlinerestimate-larger}
		\begin{aligned}
			\|\bar{\chi}_TZ^{\leq 2} [F(\partial \tilde{u})]\|_{L_t^1L_x^2}&\lesssim T\|\bar{\chi}_TZ^{\leq 2} [F(\partial \tilde{u})]\|_{L_t^\infty L_x^2}\\
			&\lesssim T\rho(C_S \|\tilde{u}\|_{E_2} {\<T\>}^{-1}){\<T\>}^{-1}\|\tilde{u}\|_{E_2}
			^2\\
			&\lesssim \rho(C_S \|\tilde{u}\|_{E_2})\|\tilde{u}\|_{E_2}
			^2.
		\end{aligned}
	\end{equation}
	Combining \eqref{eq-proof-thm-3dim-nonlinerestimate}, \eqref{eq-proof-thm-3dim-nonlinerestimate-smaller}, \eqref{eq-proof-thm-3dim-nonlinerestimate-smaller-Tinfty}, and \eqref{eq-proof-thm-3dim-nonlinerestimate-larger}, we deduce that
	\begin{equation}\label{eq-proof-thm-3dim-nonlinerestimate-final}
		\|Z^\alpha [F(\partial \tilde{u})]\|_{LE^*+L_t^1L_x^2}\lesssim\Lambda(T,\|\tilde{u}\|_{E_2})\|\tilde{u}\|_{E_2}\|\tilde{u}\|_{E_2\cap LE_2}.
	\end{equation}
	Finally, using \eqref{eq-proof-thm-3dim-2orderenergyiteration} and \eqref{eq-proof-thm-3dim-nonlinerestimate-final}, we obtain the estimates \eqref{eq-prop-3dim-2regular-fbound}.
	
	As for the estimate \eqref{eq-prop-3dim-2regular-fdiffer},
	\begin{equation}\label{eq-proof-thm-3dim-2orderenergydiffer}
		\begin{aligned}
			\|I[\tilde{u}]-I[\tilde{v}]\|_{LE_1\cap E_1}&\lesssim\|Z^{\leq1} [F(\partial \tilde{u})-F(\partial \tilde{v})]\|_{LE^*+L_t^1L_x^2}\\
			&\quad+\|F(\partial \tilde{u})-F(\partial \tilde{v})\|_{(L_t^\infty\cap L_t^2)L_x^2(B_{2R})}. \\
		\end{aligned}
	\end{equation}
	We denote $\tilde{w}=\tilde{u}-\tilde{v}$ and use $P$ to stand for $P(\tilde{u},\tilde{v})$. Noting that
	\begin{equation*}
		\begin{gathered}
			|F(\partial \tilde{u})-F(\partial \tilde{v})|\lesssim \rho(C_S P {\<r\>}^{-1}){\<r\>}^{-1}P|\partial \tilde{w}|, \\
			|Z[F(\partial \tilde{u})-F(\partial \tilde{v})]|\lesssim \rho(C_S P {\<r\>}^{-1}){\<r\>}^{-1}(P|Z\partial \tilde{w}|+\<r\>|Z\partial \tilde{v}||\partial \tilde{w}|),\\
		\end{gathered}
	\end{equation*}
	we have
	\begin{gather}
		\|F(\partial \tilde{u})-F(\partial \tilde{v})\|_{(L_t^\infty\cap L_t^2)L_x^2(B_{2R})}\lesssim \rho(C_S P)P\|\tilde{w}\|_{LE\cap E},\label{eq-proof-thm-3dim-lowerorderenergydifer}\\
		\begin{aligned}
			\|\bar{\chi}_TZ^{\leq 1} [F(\partial \tilde{u})-F(\partial \tilde{v})]\|_{L_t^1L_x^2}&\lesssim T\|\bar{\chi}_TZ^{\leq 1} [F(\partial \tilde{u})-F(\partial \tilde{v})]\|_{L_t^\infty L_x^2}\\
			&\lesssim T\rho(C_S P {\<T\>}^{-1}){\<T\>}^{-1}P\|\tilde{w}\|_{E_1}\\
			&\lesssim \rho(C_S P )P\|\tilde{w}\|_{E_1}. 
		\end{aligned}\label{eq-proof-thm-3dim-nonlinerestimatediffer-larger}
	\end{gather}
	Also,
	\begin{equation}\label{eq-proof-thm-3dim-nonlinerestimatediffer-smaller}
		\begin{aligned}
			&\quad \|\chi_TZ^{\leq1} [F(\partial \tilde{u})-F(\partial \tilde{v})]\|_{LE^*}\\
			&\lesssim\sum_{j=0}^{2^{j-1}<\<T\>}2^{j\half}\|\Phi_j(x) \rho(C_S P{\<r\>}^{-1})\left({\<r\>}^{-1}P|Z^{\leq1}\partial \tilde{w}|+|Z\partial \tilde{v}||\partial\tilde{w}|\right)\|_{L_t^2L_x^2}\\
			&\lesssim {\left(\rho(C_S P)+\int_{{\<T\>}^{-1}}^{1}\frac{\rho(C_S P\tau)}{\tau}d\tau\right)} P\|\tilde{w}\|_{LE_1\cap E_1}\\
		\end{aligned}
	\end{equation}
	and, if $\int_{0}^{1}\frac{\rho(\tau)}{\tau}d\tau<\infty$,
	\begin{equation}\label{eq-proof-thm-3dim-nonlinerestimatediffer-smaller-Tinfty}
		\|\chi_TZ^{\leq1} [F(\partial \tilde{u})-F(\partial \tilde{v})]\|_{LE^*}\lesssim{\left(\rho(C_S P)+\int_{0}^{1}\frac{\rho(C_S P\tau)}{\tau}d\tau\right)} P\|\tilde{w}\|_{LE_1\cap E_1}.
	\end{equation}
	Hence,  combining \eqref{eq-proof-thm-3dim-2orderenergydiffer}, \eqref{eq-proof-thm-3dim-lowerorderenergydifer}, \eqref{eq-proof-thm-3dim-nonlinerestimatediffer-smaller}, \eqref{eq-proof-thm-3dim-nonlinerestimatediffer-smaller-Tinfty}, and \eqref{eq-proof-thm-3dim-nonlinerestimatediffer-larger}, we obtain the estimates \eqref{eq-prop-3dim-2regular-fdiffer}.
\end{proof}

Using Proposition \ref{prop-proof-3dim-iteration}, we can finish the proof of Theorem \ref{thm-3dim}.

\begin{proof}[Proof of Theorem \ref{thm-3dim}]
The uniqueness is obvious. We just focus on the existence and the lower bound of the lifespan.

Consider the following Dirichlet-wave equation
\begin{equation}\label{eq-initialiteration}
	\left\{\begin{aligned}
		&\Box u(t,x)=[P_0(x)+P_1(x)t]\theta(t)&&, \ (t,x)\in (0,T)\times M, \ \\
		&u(t,x)=0&&, \ x\in \partial M,  \ t>0 , \\
		&u(0,x)=f(x), \ u_t(0,x)=g(x) &&, \ x\in M,
	\end{aligned}
	\right.
\end{equation}
where
\begin{equation*}
	\begin{aligned}
		P_j(x)=-\Delta[\Psi_j(\nabla^{\leq j} f,\nabla^{\leq j-1}g)](x)+\Psi_{j+2}(\nabla^{\leq j+2} f,\nabla^{\leq j+1}g)(x), \ j=0, 1,
	\end{aligned}
\end{equation*}
and $\theta(t)$ is a cut-off function such that $\theta(0)\equiv1$ in some neighborhood of $0$.
By a direct calculation, we find that the equation \eqref{eq-initialiteration} has the same compatibility condition as that of the equation \eqref{eq-critnonlin-M}. 
By Proposition \ref{prop-localenergy}, there exists a constant $C_0>1$ such that the solution $u_0$ to the equation \eqref{eq-initialiteration} satisfying
\begin{equation*}
	\|u_0\|_{LE_{2}\cap E_2}\leq C_0C_1\ep.
\end{equation*}
As long as we choose $\ep$ and $T$ such that
\begin{equation}\label{eq-prof-thm-3dim-ep-T-choose}
	\left\{
	\begin{gathered}
		C_0C_1\ep\leq\frac{1}{2C_S},\\
		(C_0C_1)^2\ep\Lambda(T,C_0C_1\ep)\leq C_0-1,\\
		C_0C_1C_2\ep\Lambda(T,C_0C_1\ep)\leq\half,
	\end{gathered}
	\right.
\end{equation}
we can take $u_j=I[u_{j-1}]$, $j\in\N^+$, and deduce that, by \eqref{eq-prop-3dim-2regular-fbound} and \eqref{eq-prop-3dim-2regular-fdiffer}, for all $j\in\N^+$,
\begin{gather*}
	\|u_j\|_{LE_2\cap E_2}\leq C_0C_1\ep , \\
	\|u_{j+1}-u_{j}\|_{LE_{1}\cap E_1}\leq \half\|u_{j}-u_{j-1}\|_{LE_1\cap E_1}.
\end{gather*}
Thus, we find a unique solution
\begin{equation*}
	u\in L_{t,loc}^\infty H^3, \ \partial^i u\in L_t^\infty H^{3-i}\cap C_bH^{2-i}, \ 1\leq i\leq2,
\end{equation*}
with $\|u\|_{LE_2\cap E_2}\leq C_0C_1\ep$. Strictly speaking, to complete the proof, we need also to prove the regularity of the solution, that is, $\partial^i u\in C_tH^{3-i}$, $0\leq i \leq2$. As this is standard, we omit details here and refer the reader to the end of Section 4 in \cite{MR3378835} or \cite[P533]{MR2980460}.

To conclude this subsection, we discuss the global existence and the lower bound estimate of the lifespan.
Obviously, if $\int_{0}^{1}\frac{\rho(\tau)}{\tau}d\tau<\infty$, we can choose $T=\infty$ in \eqref{eq-prof-thm-3dim-ep-T-choose}, that is,
\begin{equation*}
	\left\{
	\begin{gathered}
		C_0C_1\ep\leq\frac{1}{2C_S},\\
		(C_0C_1)^2\ep\Lambda(\infty,C_0C_1\ep)\leq C_0-1,\\
		C_0C_1C_2\ep\Lambda(\infty,C_0C_1\ep)\leq\half.
	\end{gathered}
	\right.
\end{equation*}
Due to $\lim_{t\rightarrow0}\rho(\tau)=0$, there exists an $\ep_1$ such that for all $\ep\in(0,\ep_1)$, the problem \eqref{eq-critnonlin-M} has a unique global solution. Otherwise, we choose $\ep$ satisfying
\begin{equation*}
	\left\{
	\begin{gathered}
		C_0C_1\ep\leq\frac{1}{2C_S},\\
		(C_0C_1)^2\ep\rho(C_SC_0C_1\ep)\leq\frac{C_0-1}{2},\\
		C_0C_1C_2\ep\rho(C_SC_0C_1\ep)\leq\frac{1}{4}.
	\end{gathered}
	\right.
\end{equation*}
Denote
\begin{equation*}
	\tilde{c}_1=\min\left\{\frac{C_0-1}{2(C_0C_1)^2},\frac{1}{4C_0C_1C_2}\right\}, \ \tilde{c}_2=C_SC_0C_1,
\end{equation*}
and, to fulfill the requirement \eqref{eq-prof-thm-3dim-ep-T-choose}, take $T_\ep$ such that 
\begin{equation*}
	\tilde{c}_1\ep^{-1}=\int_{{\<T_\ep\>}^{-1}}^{1}\frac{\rho(\tilde{c}_2\ep\tau)}{\tau}d\tau=\mathcal{H}_{2,3}(\tilde{c}_2\ep{\<T_\ep\>}^{-1})-\mathcal{H}_{2,3}(\tilde{c}_2\ep),
\end{equation*}
that is,
\begin{equation*}
	\<T_\ep\>=\tilde{c}_2\ep[\mathcal{H}_{2,3}^{inv}(\tilde{c}_1\ep^{-1}+\mathcal{H}_{2,3}(\tilde{c}_2\ep))]^{-1}.\qedhere
\end{equation*}
\end{proof}

\section{Proofs of Remark \ref{coro-Euclidgeq3} and Theorem \ref{coro-Euclid2}}
The proofs of Remark \ref{coro-Euclidgeq3} and Theorem \ref{coro-Euclid2} are established by (local) energy estimates and the control of the $L^\infty$ norm of $\partial u$. In this section, we give sketches of the proofs.

\subsection{Proof of Remark \ref{coro-Euclidgeq3}}
In this subsection, $E_i$ and $LE_i$ stand for $E_{i,Z}$ and $LE_{i,Z}$. This time the solution will be obtained by iteration in space $S_{\kappa,Z,\infty}$.

To obtain the $\kappa$th-order version estimates of \eqref{eq-prop-3dim-2regular-fbound} and \eqref{eq-prop-3dim-2regular-fdiffer}, that is, there exists constants $C_3$ and $C_4$ such that, for all $\tilde{u},\tilde{v}\in S_\infty(1/(2C_S))$,
\begin{gather*}
	\|I[\tilde{u}]\|_{LE_\kappa\cap E_\kappa}\leq C_3\ep+C_3\|\tilde{u}\|_{E_\kappa}^{\kappa-1}\|\tilde{u}\|_{LE_\kappa\cap E_\kappa} , \\
	\|I[\tilde{u}]-I[\tilde{v}]\|_{LE_{\kappa-1}\cap E_{\kappa-1}}\leq C_4P_\kappa(\tilde{u},\tilde{v})^{\kappa-1}\|\tilde{u}-\tilde{v}\|_{LE_{\kappa-1}\cap E_{\kappa-1}},
\end{gather*}
where $P_\kappa(\tilde{u},\tilde{v})=\max\{\|\tilde{u}\|_{LE_\kappa\cap E_\kappa},\|\tilde{v}\|_{LE_\kappa\cap E_\kappa}\}$, we only need to verify
\begin{equation}\label{eq-proof-coro-Euclidgeq3-nonlinearcontrol}
	\|Z^{\leq\kappa} [F(\partial \tilde{u})]\|_{LE^*}\lesssim\|\tilde{u}\|_{E_\kappa}^{\kappa-1}\|\tilde{u}\|_{LE_\kappa\cap E_\kappa}.
\end{equation}
By the chain rule, for $\partial \tilde{u}\in C_tH^{\kappa}(M)\cap C_t^\kappa L_x^2(M)$,
\begin{equation}\label{eq-proof-coro-Euclidgeq3-chainrule}
	\begin{aligned}
		&\quad \|Z^{\leq \kappa} [F(\partial \tilde{u})]\|_{LE^*}\\
		&\lesssim\|F(\partial \tilde{u})\|_{LE^*}+\\
		&\quad \rho(C_S\|\tilde{u}\|_{E_\kappa}) \sum_{j=0}^{\infty}2^{j\half}\sum_{\substack{1\leq\mu\leq \kappa\\ 1\leq b_1+\dots+b_\mu\leq\kappa}}\|\Phi_j(x) \left(|\partial \tilde{u}|^{\kappa-\mu}|Z^{b_1}\partial \tilde{u}|\cdots|Z^{b_\mu}\partial \tilde{u}|\right)\|_{L_t^2L_x^2}\\
		&\lesssim\|F(\partial \tilde{u})\|_{LE^*}+\\
		&\quad \sum_{j=0}^{\infty}\sum_{\substack{1\leq\mu\leq \kappa\\ 1\leq b_1+\dots+b_\mu\leq\kappa}}2^{-j\frac{(n-1)(\kappa-\mu)-1}{2}}\|\tilde{u}\|_{E_\kappa}^{\kappa-\mu}\|\Phi_j(x) \left(|Z^{b_1}\partial \tilde{u}|\cdots|Z^{b_\mu}\partial \tilde{u}|\right)\|_{L_t^2L_x^2}
	\end{aligned}
\end{equation}
Noticing that
\begin{equation}\label{eq-proof-coro-Euclidgeq3-Linfty}
	|Z^{b_i}\partial \tilde{u}|\lesssim\<r\>^{-\frac{n-1}{2}}\|\tilde{u}\|_{E_\kappa}, \ \text{if} \ b_i+\frac{n}{2}<\kappa,
\end{equation}
we can focus on the terms in the last line of \eqref{eq-proof-coro-Euclidgeq3-chainrule} where $\mu\geq2$ and $b_i+\frac{n}{2}\geq\kappa$, $\forall 1\leq i\leq\mu$. 
Because 
\begin{equation*}
	\frac{\mu-1}{2}-\frac{(\mu-1)\kappa-\sum_{1}^{\mu-1}b_i}{n}<\half, \ \text{for} \ \mu\geq2, \ \kappa>\frac{n}{2},
\end{equation*}
it is always possible for us to choose $\{q_i\}_{1\leq i\leq \mu}$ such that
\begin{equation*}
	\left\{
	\begin{aligned}
		&\frac{1}{q_i}=\left(\half-\frac{\kappa-b_i}{n}\right)_+, \ 1\leq i \leq \mu-1,\\
		&\frac{1}{q_\mu}=\half-\sum_{i=1}^{\mu-1}\frac{1}{q_i},
	\end{aligned}\right.
\end{equation*}
Then, it follows that, by Lemma \ref{lem-Sobolev-rotaition},
\begin{equation}\label{eq-proof-coro-Euclidgeq3-Lq}
	\begin{aligned}
		&\quad\|\Phi_j(x) \left(|Z^{b_1}\partial \tilde{u}|\cdots|Z^{b_\mu}\partial \tilde{u}|\right)\|_{L_t^2L_x^2}\\
		&= \|\Phi_j(x)\<r\>^{-\frac{(n-1)(\mu-1)}{2}}\prod_{i=1}^{\mu} \left(\<r\>^{(n-1)(\half-\frac{1}{q_i})}|Z^{b_i}\partial \tilde{u}|\right)\|_{L_t^2L_x^2}\\
		&\lesssim2^{-j\frac{(n-1)(\mu-1)}{2}}\|\tilde{u}\|_{E_\kappa}^{\kappa-1}\|Y^{\leq\kappa-b_\mu}[\Phi_j(x)Z^{b_\mu}\partial \tilde{u}]\|_{L_t^2L_x^2}\\
		&\lesssim2^{-j\frac{(n-1)(\mu-1)}{2}}\|\tilde{u}\|_{E_\kappa}^{\kappa-1}\|[\Phi_{j-1}(x)+\Phi_j(x)+\Phi_{j+1}(x)]Z^{\leq\kappa}\partial \tilde{u}\|_{L_t^2L_x^2}.
	\end{aligned}
\end{equation}
Hence, combining \eqref{eq-proof-coro-Euclidgeq3-chainrule}, \eqref{eq-proof-coro-Euclidgeq3-Linfty}, and \eqref{eq-proof-coro-Euclidgeq3-Lq}, we establish the desired inequality \eqref{eq-proof-coro-Euclidgeq3-nonlinearcontrol}.

Let $u_0$ be the solution to the following Dirichlet-wave equation
\begin{equation}\label{eq-proof-coro-Euclidgeq3-initialiter}
	\left\{\begin{aligned}
		&\Box u(t,x)=\left[\sum_{j=0}^{\kappa-1}\frac{t^j}{j!}P_j(x)\right]\theta(t)&&, \ (t,x)\in (0,T)\times M, \ \\
		&u(t,x)=0&&, \ x\in \partial M,  \ t>0 , \\
		&u(0,x)=f(x), \ u_t(0,x)=g(x) &&, \ x\in M,
	\end{aligned}
	\right.
\end{equation}
where
\begin{equation*}
	\begin{aligned}
		P_j(x)=-\Delta[\Psi_j(\nabla^{\leq j} f,\nabla^{\leq j-1}g)](x)+\Psi_{j+2}(\nabla^{\leq j+2} f,\nabla^{\leq j+1}g)(x), \ 0\leq j\leq\kappa-1,
	\end{aligned}
\end{equation*}
and $\theta(t)$ is a cut-off function such that $\theta(0)\equiv1$ in some neighborhood of $0$. 
By a direct calculation, we find that the equation \eqref{eq-proof-coro-Euclidgeq3-initialiter} has the same compatibility condition as that of the equation \eqref{eq-critnonlin-M}. 
As long as $\ep$ is small enough, the iteration sequence, $u_j=I[u_{j-1}]$, $j\in\N^+$, will converge to the unique global solution to the problem \eqref{eq-critnonlin-M}.

\subsection{Proof of Theorem \ref{coro-Euclid2}}
In this subsection, $E_i$ and $S_T$ stand for $E_{i,\Gamma}$ and $S_{\kappa,\Gamma,T}$.

When $n=2$, due to the lack of local energy estimate, the problem \eqref{eq-critnonlin-M} is suspended for non-empty obstacles. For $M=\R^2$ and $\kappa\geq2$, standard energy estimates and Klainerman-Sobolev inequalities are enough to bring out the expected results.

By energy estimates, for all $\tilde{u}\in S_T(1/(2C_{KS})$, it follows
\begin{equation}\label{eq-prof-coro-Euclid2-itera-enrgy}
	\begin{aligned}
		&\quad\|I[\tilde{u}]\|_{E_\kappa}\\
		&\lesssim \ep+\|\Gamma^{\leq\kappa}[F(\partial\tilde{u})]\|_{L_t^1L_x^2}\\
		&\lesssim\ep+\|F(\partial\tilde{u})\|_{L_t^1L_x^2}+\sum_{\substack{1\leq\mu\leq \kappa\\ 1\leq b_1+\dots+b_\mu\leq\kappa}} \||\partial^\mu F(\partial \tilde{u})||\Gamma^{b_1}\partial \tilde{u}|\cdot\cdots\cdot|\Gamma^{b_\mu}\partial \tilde{u}|\|_{L_t^1L_x^2}.\\
	\end{aligned}
\end{equation}
Just as in the poof of Remark \ref{coro-Euclidgeq3}, we only need to focus on terms where $\mu\geq2$ and $b_i+1\geq\kappa$, $\forall 1\leq i\leq\mu$.
Then, we have $\sum_{i=1}^{\mu}b_i+\mu\geq\kappa\mu \Rightarrow \mu\leq\frac{\kappa}{\kappa-1}$. Due to $\kappa\geq2$, we only need to investigate terms with $\kappa=\mu=2$, that is,
\begin{equation}\label{eq-prof-coro-Euclid2-kappa-mu-2}
	\begin{aligned}
		&\quad\||\partial^2 F(\partial \tilde{u})||\Gamma\partial \tilde{u}|^2\|_{L_x^2}\\
		&\leq\|\chi_t|\partial^2 F(\partial \tilde{u})||\Gamma\partial \tilde{u}|^2\|_{L_x^2}+\|\bar{\chi}_t|\partial^2 F(\partial \tilde{u})||\Gamma\partial \tilde{u}|^2\|_{L_x^2}\\
		&\lesssim\rho(C_{KS}\<t\>^{-\frac{1}{2}}\|\tilde{u}\|_{E_2})\Big(\|\chi_t\Gamma\partial \tilde{u}\|_{{L_x^4}}^2+\<t\>^{-\half}\|(|x|^\half|\Gamma\partial \tilde{u}|)|\Gamma\partial \tilde{u}|\|_{L_x^2}\Big)\\
		&\lesssim\rho(C_{KS}\<t\>^{-\frac{1}{2}}\|\tilde{u}\|_{E_2})\Big(\<t\>^{-1}\|\tilde{u}\|_{E_2}^{2}+\<t\>^{-\half}\|(|x|^\half|\Gamma\partial \tilde{u}|)\|_{L_r^\infty L_\omega^4}\|\Gamma\partial \tilde{u}\|_{L_r^2 L_\omega^4}\Big)\\
		&\lesssim\rho(C_{KS}\<t\>^{-\frac{1}{2}}\|\tilde{u}\|_{E_2})\<t\>^{-\frac{1}{2}}\|\tilde{u}\|_{E_2}^{2},
	\end{aligned}
\end{equation}
where we used the Proposition \ref{prop-KSinequality}, Proposition \ref{prop-traceestimate-Rn}, and Sobolev embedding on $\mathbb{S}^1$. Combining \eqref{eq-prof-coro-Euclid2-itera-enrgy} and \eqref{eq-prof-coro-Euclid2-kappa-mu-2}, one deduces that
\begin{equation*}
		\|I[\tilde{u}]\|_{E_\kappa}
		\lesssim \ep+\int_{0}^{T}\rho(C_{KS}\|\tilde{u}\|_{E_\kappa}\<t\>^{-\half})\<t\>^{-\frac{\kappa-1}{2}}dt \cdot \|\tilde{u}\|_{E_\kappa}^\kappa.
\end{equation*}
Also, for all $\tilde{u}, \tilde{v}\in S_T(1/(2C_{KS})$ , one has
\begin{equation*}
	\|I[\tilde{u}]-I[\tilde{v}]\|_{E_{\kappa-1}}\lesssim \int_{0}^{T}\rho(C_{KS}W_\kappa(\tilde{u},\tilde{v})\<t\>^{-\half})\<t\>^{-\frac{\kappa-1}{2}}dt \cdot W_\kappa(\tilde{u},\tilde{v})^{\kappa-1}\|\tilde{u}-\tilde{v}\|_{E_{\kappa-1}},
\end{equation*}
where $W_\kappa(\tilde{u},\tilde{v})=\max\{\|\tilde{u}\|_{E_\kappa},\|\tilde{v}\|_{E_\kappa}\}$. Thus, if $\int_{0}^{1}\frac{\rho(\tau)}{\tau}\tau^{\kappa-3} d\tau<\infty$, there exists a unique global solution to the problem \eqref{eq-critnonlin-M}. When $\int_{0}^{1}\frac{\rho(\tau)}{\tau}\tau^{\kappa-3} d\tau=\infty$, due to 
\begin{equation*}
	\int_{s}^{1}\frac{\rho(C_{KS}\ep\tau)}{\tau}\frac{\tau^{\kappa-3}}{\sqrt{1-\tau^4}}d\tau\sim\rho(C_{KS}\ep)+\int_{s}^{1}\frac{\rho(C_{KS}\ep\tau)}{\tau}\tau^{\kappa-3}d\tau,
\end{equation*}
there exist positive constants $\tilde{c}_3$, $\tilde{c}_4$ such that we have a unique solution on $[0,T_\ep]$ with $T_\ep$ satisfying
\begin{equation*}
		\tilde{c}_3=\ep^{\kappa-1}\int_{\<T_\ep\>^{-\frac{1}{2}}}^{1}\frac{\rho(\tilde{c}_4\ep\tau)}{\tau}\tau^{\kappa-3}d\tau=\tilde{c}_4^{3-\kappa}\ep^{2}\int_{\tilde{c}_4\ep\<T_\ep\>^{-\frac{1}{2}}}^{\tilde{c}_4\ep}\frac{\rho(\tau)}{\tau}\tau^{\kappa-3}d\tau,
\end{equation*}
that is,
\begin{equation*}
	\<T_\ep\>=\tilde{c}_4^2\ep^2[\mathcal{H}_{\kappa,2}^{inv}(\tilde{c}_3\tilde{c}_4^{\kappa-3}\ep^{-2}+\mathcal{H}_{\kappa,2}(\tilde{c}_4\ep))]^{-2}.
\end{equation*}

\section{Proofs of Theorem \ref{thm-radial} and Remark \ref{coro-radial-geq4}}

In this section, $E_i$, $LE_i$, and $S_T$ stand for $E_{i,\partial}$, $LE_{i,\partial}$, and $S_{\kappa,\partial,T}$.

\subsection{Proof of Theorem \ref{thm-radial}}
Recall Proposition \ref{prop-traceestimate-M} and $M=\{x:|x|>1\}$. For a radial function $u\in H^1(M)$, there exists a constant $C_{TR}$ such that
\begin{equation*}
	|u(x)|\leq C_{TR}|x|^{-\frac{(n-1)}{2}}\|u\|_{H^1} , \ |x|>1.
\end{equation*}
Thus, following the proof of Proposition \ref{prop-proof-3dim-iteration}, we have, for $\tilde{u},\tilde{v}\in S_T(1/(2C_{TR}))$,
\begin{gather*}
	\|I[\tilde{u}]\|_{LE_1\cap E_1}\leq C_5 \ep+C_5\Lambda(T,\|\tilde{u}\|_{E_1})\|\tilde{u}\|_{LE_1\cap E_1} , \label{eq-proof-radial-1regular-fbound} \\
	\|I[\tilde{u}]-I[\tilde{v}]\|_{LE\cap E}\leq C_6\Lambda(T,P(\tilde{u},\tilde{v}))\|\tilde{u}-\tilde{v}\|_{LE\cap E},\label{eq-proof-radial-1regular-fdiffer}
\end{gather*}
where
\begin{equation*}
	\begin{aligned}
	&\quad\Lambda(T,\zeta)\\
	&=\left\{\begin{aligned}
		&\rho(C_{TR}\zeta)+\inf_{\lambda>1} \left\{\int_{\lambda^{-\frac{n-1}{2}}}^{1}\frac{\rho(C_{TR} \zeta\tau)}{\tau^{p_c(n)}}d\tau+T\rho(C_{TR} \zeta \lambda^{-\frac{n-1}{2}})\right\}&&, \ 0<T<\infty ,\\
		&\rho(C_{TR}\zeta)+\int_{0}^{1}\frac{\rho(C_{TR} \zeta\tau)}{\tau^{p_c(n)}}d\tau &&, \ T=\infty,
	\end{aligned}\right.
\end{aligned}
\end{equation*}
and $P(\tilde{u},\tilde{v})=\max
\{\|\tilde{u}\|_{LE_1\cap E_1},\|\tilde{v}\|_{LE_1\cap E_1}\}$. Once again, let $u_0$ be the solution to the following Dirichlet-wave equation
\begin{equation*}
	\left\{\begin{aligned}
		&\Box u(t,x)=P_0(x)\theta(t)&&, \ (t,x)\in (0,T)\times M, \ \\
		&u(t,x)=0&&, \ x\in \partial M,  \ t>0 , \\
		&u(0,x)=f, \ u_t(0,x)=g &&, \ x\in M,
	\end{aligned}
	\right.
\end{equation*}
where
\begin{equation*}
	P_0(x)=F(g,\partial_r f)(x),
\end{equation*}
and $\theta(t)$ is a cut-off function such that $\theta(0)\equiv1$ in some neighborhood of $0$. Then, there exists a constant $\tilde{C}_0>1$ such that 
\begin{equation*}
	\|u\|_{LE_1\cap E_1}\leq \tilde{C}_0C_5\ep.
\end{equation*}
As long as $\ep$ is small enough, the iteration sequence, $u_j=I[u_{j-1}]$, $j\in\N^+$, will converge to the unique radial solution to the problem \eqref{eq-critnonlin-radial}.

We should select $\ep$ and $T_\ep$ such that 
\begin{equation}\label{eq-prof-thm-2dim-ep-T-choose}
	\left\{
	\begin{gathered}
		\tilde{C}_0C_5\ep\leq\frac{1}{2C_{TR}},\\
		\tilde{C}_0C_5\Lambda(T_\ep,\tilde{C}_0C_5\ep)\leq \tilde{C}_0-1,\\
		C_6\Lambda(T_\ep,\tilde{C}_0C_5\ep)\leq\half.
	\end{gathered}
	\right.
\end{equation}
If $\int_{0}^{1}\frac{\rho(\tau)}{\tau}\tau^{-\frac{2}{n-1}}d\tau<\infty$, there exists a unique radial global solution to the problem \eqref{eq-critnonlin-radial}, as long as $\ep$ fulfills 
\begin{equation*}
	\left\{
	\begin{gathered}
		\tilde{C}_0C_5\ep\leq\frac{1}{2C_{TR}},\\
		\tilde{C}_0C_5\Lambda(\infty,\tilde{C}_0C_5\ep)\leq \tilde{C}_0-1,\\
		C_6\Lambda(\infty,\tilde{C}_0C_5\ep)\leq\half.
	\end{gathered}
	\right.
\end{equation*}
When $\int_{0}^{1}\frac{\rho(\tau)}{\tau}\tau^{-\frac{2}{n-1}}d\tau=\infty$, we choose $\ep$ satisfying
\begin{equation*}
	\left\{
	\begin{gathered}
		\tilde{C}_0C_5\ep\leq\frac{1}{2C_{TR}},\\
		\tilde{C}_0C_5\rho(C_{TR}\tilde{C}_0C_5\ep)\leq\frac{\tilde{C}_0-1}{3},\\
		C_6\rho(C_{TR}\tilde{C}_0C_5\ep)\leq\frac{1}{6}.
	\end{gathered}
	\right.
\end{equation*}
Recall that 
\begin{equation*}
	\mathcal{H}_{1,n}(s)=\int_{s}^{1}\frac{\rho(\tau)}{\tau}\tau^{-\frac{2}{n-1}}d\tau.
\end{equation*}
This time, we have
\begin{equation*}
	\mathcal{H}_{1,n}(s)\geq\frac{n-1}{2}\rho(s)[s^{-\frac{2}{n-1}}-1]\geq\frac{n-1}{4}\rho(s)s^{-\frac{2}{n-1}}, \ \text{for} \ 0<s\leq2^{-\frac{n-1}{2}},
\end{equation*}
and
\begin{equation*}
	\mathcal{H}_{1,n}(s)\leq \frac{n-1}{2}\|\rho\|_\infty[s^{-\frac{2}{n-1}}-1].
\end{equation*}
Denote 
\begin{equation*}
	\tilde{c}_5=\min\{\frac{\tilde{C}_0-1}{3\tilde{C}_0C_5},\frac{1}{6C_6}\}, \ \tilde{c}_6=C_{TR}\tilde{C}_0C_5.
\end{equation*}
Notice that $\mathcal{H}_{1,n}$ is decreasing and $\lim_{s\rightarrow0}\mathcal{H}_{1,n}=\infty$, we can take $\lambda_\ep$ such that 
\begin{equation*}
	\tilde{c}_5=\int_{\lambda_\ep^{-\frac{n-1}{2}}}^{1}\frac{\rho(\tilde{c}_6\ep\tau)}{\tau}\tau^{-\frac{2}{n-1}}d\tau=(\tilde{c}_6\ep)^{\frac{2}{n-1}}[\mathcal{H}_{1,n}(\tilde{c}_6\ep\lambda_\ep^{-\frac{n-1}{2}})-\mathcal{H}_{1,n}(\tilde{c}_6\ep)].
\end{equation*}
Further require that 
\begin{equation*}
	\ep\leq \tilde{c}_5^\frac{n-1}{2}\tilde{c}_6^{-1}[\mathcal{H}_{1,n}(2^{-\frac{n-1}{2}})]^{-\frac{n-1}{2}},
\end{equation*}
and then $\tilde{c}_6\ep\lambda_\ep^{-\frac{n-1}{2}}\leq 2^{-\frac{n-1}{2}}$.
Thus, it follows that
\begin{equation*}
	\begin{aligned}
		&\quad \rho\left(\tilde{c}_6\ep\lambda_\ep^{-\frac{n-1}{2}}\right)\\
		&\leq \rho\left(\mathcal{H}_{1,n}^{inv}(\tilde{c}_5\tilde{c}_6^{-\frac{2}{n-1}}\ep^{-\frac{2}{n-1}}+\mathcal{H}_{1,n}(\tilde{c}_6\ep))\right) \\
		&\leq\frac{4}{n-1}\left(\tilde{c}_5\tilde{c}_6^{-\frac{2}{n-1}}\ep^{-\frac{2}{n-1}}+\mathcal{H}_{1,n}(\tilde{c}_6\ep)\right)\left[\mathcal{H}_{1,n}^{inv}(\tilde{c}_5\tilde{c}_6^{-\frac{2}{n-1}}\ep^{-\frac{2}{n-1}}+\mathcal{H}_{1,n}(\tilde{c}_6\ep))\right]^{\frac{2}{n-1}}\\
		&\leq\frac{2(2\tilde{c}_5+(n-1)\|\rho\|_\infty)\tilde{c}_6^{-\frac{2}{n-1}}}{n-1}\ep^{-\frac{2}{n-1}}\left[\mathcal{H}_{1,n}^{inv}(\tilde{c}_5\tilde{c}_6^{-\frac{2}{n-1}}\ep^{-\frac{2}{n-1}}+\mathcal{H}_{1,n}(\tilde{c}_6\ep))\right]^{\frac{2}{n-1}}.
	\end{aligned}
\end{equation*}
Then, to fulfill the requirement \eqref{eq-prof-thm-2dim-ep-T-choose}, we choose $T_\ep$ to be
\begin{equation*}
	T_\ep=\frac{(n-1)\tilde{c}_5\tilde{c}_6^{\frac{2}{n-1}}}{2(2\tilde{c}_5+(n-1)\|\rho\|_\infty)}\ep^{\frac{2}{n-1}}\left[\mathcal{H}_{1,n}^{inv}(\tilde{c}_5\tilde{c}_6^{-\frac{2}{n-1}}\ep^{-\frac{2}{n-1}}+\mathcal{H}_{1,n}(\tilde{c}_6\ep))\right]^{-\frac{2}{n-1}},
\end{equation*}
and then there exists a unique radial solution on $[0,T_\ep]$.

\subsection{Proof of Remark \ref{coro-radial-geq4}}
The proof is similar to that of Remark \ref{coro-Euclidgeq3}.
One just need to establish the inequality for $\tilde{u}\in S_\infty(1/(2C_{TR}))$
\begin{equation}\label{eq-proof-radial-geq4-chainrule}
	\|\partial^{\leq\kappa} [F(\partial_t \tilde{u},\partial_r \tilde{u})]\|_{LE^*}\lesssim\|\tilde{u}\|_{E_\kappa}^{\kappa-1}\|\tilde{u}\|_{LE_\kappa\cap E_\kappa}.
\end{equation}
Notice that for a radial function $w\in H^{\kappa}(M)$,
\begin{equation*}
	|\partial_x^\alpha w|\lesssim|\partial_r^{\leq|\alpha|}w|\lesssim |x|^{-\frac{n-1}{2}} \|\partial_r^{\leq|\alpha|}w\|_{H^1}\lesssim|x|^{-\frac{n-1}{2}} \|\partial^{\leq|\alpha|}w\|_{H^1} , \ |x|>1 , \ |\alpha|\leq\kappa;
\end{equation*}
see Li-Zhou \cite[Lemma 3.1.7 and Lemma 3.4.2]{MR3729480}.
Therefore, following the calculation in \eqref{eq-proof-coro-Euclidgeq3-chainrule}, we obtain the estimate \eqref{eq-proof-radial-geq4-chainrule}.

\section{Proof of Theorem \ref{thm-lifespan-sharp}}

Due to the finite speed of propagation, the semilinear wave equation \eqref{eq-critnonlin-M} can be localized. 
Since, for the problem posed on $\R^n$, blow-up can be showed by constructing an integral near the wave front and derive an ordinary differential inequality, we can use this argument to deduce the blow-up result and provide the upper bound estimates of the lifespan. This argument appears in many former works, e.g., \cite{MR0879355,MR1845748,2306.11478}.

\begin{proof}[Proof of Theorem \ref{thm-lifespan-sharp}]
	According Theorem \ref{thm-radial} and Theorem \ref{coro-Euclid2}, for sufficiently small $\ep$, the lifespan $T(\ep)$ of the problem \eqref{eq-blowup-upperbound-n} has the lower bound \eqref{eq-thm-radial-lowerboundestimate} or \eqref{eq-coro-Euclid2-lowerboundestimate} with $\tau^{1+a-\kappa}\mu(\tau)$ replacing the $\rho(\tau)$ in $\mathcal{H}_{\kappa,n}$. For each $\tilde{T}\in(0,T(\ep))$, the solution $u\in C([0,\tilde{T}];H^2(M))\cap C^1([0,\tilde{T}];H^1(M))$ is radial, hence $\partial_{t}u$ is radial and continuous, and $\|\partial_t u\|_\infty\leq1$ on $[0,\tilde{T}]$. By finite speed of propagation, we deduce that for all $t>0$, $\supp u(t,\cdot) \subset \{x\in\R^n : |x|\leq t+3\}$; see the argument in \cite[Lemma 2.11]{2211.01594}. 
	
	Let a linear operator $*:A(t,x)\mapsto A^*(t,x_1)$ be
	\begin{equation*}
		A^*(t,x_1)=\int_{\R^{n-1}}A(t,x_1,\tilde{x})d\tilde{x}  .
	\end{equation*}
	The operator $*$ is defined for all admissible functions.
	Therefore, $u^*$ is the weak solution to
	\begin{equation*}
		\left\{\begin{aligned}
			&(\partial_{t}^2-\partial_{x_1}^2) u^*(t,x_1)=F_a(\partial_tu)^*(t,x_1)&&, \ (t,x_1)\in (0,T(\ep))\times (1.5,\infty), \ \\
			&u^*(0,x_1)=\ep\phi^*(x_1), \ \partial_tu^*(0,x_1)=0 &&, \ x_1\in (1.5,\infty).
		\end{aligned}
		\right.
	\end{equation*}
	Noticing that $F_a(\partial_tu)$ is continuous, we can use d'Alembert's formula and obtain that, for $z\geq2$
	\begin{equation*}
		u^*(z-2,z)=\frac{\ep(\phi^*(2z-2)+\phi^*(2))}{2}+\half\int_{0}^{z-2}ds\int_{2+s}^{2z-2-s}F_a(\partial_tu)^*(s,x_1)dx_1,
	\end{equation*}
	which implies $u^*(z-2,z)\geq0$; see the argument and details in \cite[Section 4]{MR4819613}. 
	Let $\mathcal{U}(z)=u^*(z-2,z)$. 
	Recall that $\supp u(t,\cdot) \subset \{x\in\R^n : |x|\leq t+3\}$. It follows that, for $z\geq3$,
	\begin{equation}\label{eq-blowup-onedim-u-inequal}
		\mathcal{U}(z)\geq C_\phi\ep+\half\int_{3}^{z}dy\int_{y-3}^{y-2}\int_{|\tilde{x}|\leq\sqrt{(s+3)^2-y^2}}F_a(\partial_tu(s,y,\tilde{x}))d\tilde{x}ds,
	\end{equation}
	where $C_\phi=\half\int_{\R^{n-1}}\phi(2,\tilde{x}) d\tilde{x}$.
	Noting that $\|\partial_t u\|_\infty\leq1$ on $[0,\tilde{T}]$, we obtain that
	\begin{equation*}
		\begin{aligned}
			\mathcal{U}(z)&=u^*(z-3,z)+\int_{z-3}^{z-2}\int_{|\tilde{x}|\leq\sqrt{(s+3)^2-z^2}}\partial_tu(s,z,\tilde{x})d\tilde{x}ds\\
			&\leq \int_{z-3}^{z-2}((s+3)^2-z^2)^\frac{n-1}{2}ds\\
			&=\int_{0}^{1}[(s+2z)s]^\frac{n-1}{2}ds\coloneq M(z),
		\end{aligned}
	\end{equation*}
	and, by Jensen's inequality and the convexity of $F_a(\cdot)$,
	\begin{equation}\label{eq-blowup-onedim-u-inequal-M(y)}
		\begin{aligned}
			&\quad\frac{1}{M(y)}\int_{y-3}^{y-2}\int_{|\tilde{x}|\leq\sqrt{(s+3)^2-y^2}}F_a(\partial_tu(s,y,\tilde{x}))d\tilde{x}ds\\
			&\geq F_a\left(\frac{1}{M(y)}\int_{y-3}^{y-2}\int_{|\tilde{x}|\leq\sqrt{(s+3)^2-y^2}}\partial_tu(s,y,\tilde{x})d\tilde{x}ds\right)\\
			&=F_a\left(\frac{1}{M(y)}\mathcal{U}(y)\right).
		\end{aligned}
	\end{equation}
	Combining \eqref{eq-blowup-onedim-u-inequal} and \eqref{eq-blowup-onedim-u-inequal-M(y)}, we have
	\begin{equation*}
		\begin{aligned}
			\mathcal{U}(z)&\geq C_\phi\ep+\half\int_{3}^{z}M(y)F_a\left(\frac{1}{M(y)}\mathcal{U}(y)\right)dy \\
			&\geq C_\phi\ep+\half\int_{3}^{z}M^{-a}(y)\mathcal{U}^{1+a}(y)\mu\left(\frac{1}{M(y)}\mathcal{U}(y)\right)dy\\
			&\geq C_\phi\ep+\half\int_{3}^{z}3^{-a\frac{n-1}{2}}y^{-a\frac{n-1}{2}}\mathcal{U}^{1+a}(y)\mu\left(3^{-\frac{n-1}{2}}y^{-\frac{n-1}{2}}\mathcal{U}(y)\right)dy,
		\end{aligned}
	\end{equation*}
	where $m=3^\frac{n-1}{2}$.
	Following the argument in \cite[Subsection 3.2]{2306.11478}, we deduce that
	\begin{equation}\label{eq-blowup-prop-prof-upperbound}
		\tilde{T}\leq \tilde{c}_9\ep^\frac{2}{n-1}\left[\mathcal{H}_{1,n}^{inv}(\tilde{c}_7\ep^{-\frac{2}{n-1}}+\mathcal{H}_{1,n}(\tilde{c}_8\ep))\right]^{-\frac{2}{n-1}},
	\end{equation}
	with
	\begin{equation*}
			\tilde{c}_7=\frac{3(n-1)}{a C_\phi^\frac{2}{n-1}}, \
			\tilde{c}_8=C_\phi3^{-(n-1)}, \
			\tilde{c}_9=3^{-1}C_\phi^\frac{2}{n-1}.
	\end{equation*}
	Since \eqref{eq-blowup-prop-prof-upperbound} is valid for any $\tilde{T}\in(0,T(\ep))$, we obtain the sharp upper bound estimates for the lifespan $T(\ep)$.
\end{proof}

\section{Appendix: Regularity in Proposition \ref{prop-localenergy}}\label{section-regularity}
When $\mathcal{K}=\emptyset$, the result is standard by dense argument for $(f,g,F)$. Hereby, the obstacle $\mathcal{K}$ will be non-empty.

Once we obtain the solution $u$ to \eqref{eq-Diriwave-M}, the derivatives of $u$ exist in the distributional sense, belonging to $D'((0,\infty)\times M)$. In this appendix, we give the proof for the regularity result of first-order local energy estimates. For higher order estimates, one can use induction to verify the results. 
\begin{proposition}\label{prop-regularity}
	Let $(f,g)\in H_D^{2}\times H_D^1$, $G\in C_t((H^{-1}(M))_{w*})\cap C_b(L_{loc}^{2}(M))$ with $Z^{\leq 1}G\in LE^*+L_t^1L_x^2$, and $(f,g,G)$ fulfill the compatibility condition of order $2$. Then, $\partial_t^i u\in C_b(H_D^{2-i})$, $1\leq i \leq 2$, and $Yu\in C_b(H^1)$, $\partial_tYu\in C_b(L_x^2)$.
\end{proposition}

\begin{proof}
	Let $w\in C_c^\infty(\overline{(0,\infty)\times M})$ and $w(t,\cdot)|_{\partial M}=0$, $\forall t\geq0$. Recall that $u\in C_b([0,\infty);\dot{H}_D^1(M))$, $\partial_t u\in C_b([0,\infty);L_x^2(M))$. Then, we have
	\begin{equation}\label{eq-proof-prop-regularity-inte-1st}
		\begin{aligned}
			&\quad\int_{0}^{\infty}\<u(t),\Box w(t)\>_x dt\\
			&=-\<f,\partial_t w(0)\>_x-\int_{0}^{\infty}\<\partial_tu(t),\partial_t w(t)\>_x dt+\int_{0}^{\infty}\< \nabla u(t)\cdot\nabla w(t)\>_x dt.
		\end{aligned}
	\end{equation}
	Fix a function $\eta(s)\in C_c^\infty(\R)$ with $\eta \equiv1$ on $|s|\leq\half$ and $\eta \equiv0$ on $|s|\geq1$. We denote $\eta_{\delta}=\eta(s/\delta)$, $\delta>0$ and denote 
	\begin{equation*}
		\tilde{\eta}_\delta(x)=\left[\chi_{\{y:\text{dist}(y,\mathcal{K})\leq\frac{3}{2}\delta\}}(\cdot)\right]\ast\left[\eta_{\frac{\delta}{2}}(|\cdot|)\right](x).
	\end{equation*}
	Due to the continuity, it follows that
	\begin{equation}\label{eq-proof-prop-regularity-limit}
		\begin{aligned}
			&\quad-\int_{0}^{\infty}\<\partial_tu(t),\partial_t w(t)\>_x dt+\int_{0}^{\infty}\< (\nabla u(t),\nabla w(t))\>_x dt\\
			&=\lim_{\delta_2\rightarrow0}\lim_{\delta_1\rightarrow0} -\int_{0}^{\infty}\<\eta_{\delta_1,\delta_2}(t)\partial_tu(t),\partial_t w(t)\>_x dt+\int_{0}^{\infty}\< \eta_{\delta_1,\delta_2}(t)\nabla u(t)\cdot\nabla w(t)\>_x dt,
		\end{aligned}
	\end{equation}
	where $\eta_{\delta_1,\delta_2}(t,x)=(1-\eta_{\delta_1}(t))(1-\tilde{\eta}_{\delta_2}(x))$.
	Hence, 
	\begin{equation}\label{eq-proof-prop-regularity-delta1-delta2}
		\begin{aligned}
			&\quad-\int_{0}^{\infty}\<\eta_{\delta_1,\delta_2}(t)\partial_tu(t),\partial_t w(t)\>_x dt+\int_{0}^{\infty}\< \eta_{\delta_1,\delta_2}(t)\nabla u(t)\cdot\nabla w(t)\>_x dt\\
			&=\<\eta_{\delta_1,\delta_2}G, w\>_{t,x}-\int_{0}^{\infty}\eta_{\delta_1}'(t)\<(1-\tilde{\eta}_{\delta_2})\partial_tu(t), w(t)\>_x dt\\
			&\quad+\int_{0}^{\infty}(1-\eta_{\delta_1}(t)) \<\nabla\tilde{\eta}_{\delta_2}\cdot\nabla u(t) w(t)\>_x dt\\
			&=\<\eta_{\delta_1,\delta_2}G, w\>_{t,x}+\<(1-\tilde{\eta}_{\delta_2})g,w(0)\>_x\\
			&\quad+\int_{0}^{\delta_1}(-\eta_{\delta_1}'(t))\left[\<(1-\tilde{\eta}_{\delta_2})\partial_tu(t), w(t)\>_x-\<(1-\tilde{\eta}_{\delta_2})g,w(0)\>_x\right] dt\\
			&\quad+\int_{0}^{\infty}(1-\eta_{\delta_1}(t)) \<\nabla\tilde{\eta}_{\delta_2}\cdot\nabla u(t), w(t)\>_x dt\\
		\end{aligned}
	\end{equation}
	By $w\in C_c^\infty(\overline{(0,\infty)\times M})$ and $G\in LE^*+L_t^1L_x^2$, $G\in C_t((H^{-1}(M))_{w*})$, we have
	\begin{equation}\label{eq-proof-prop-regularity-Gw}
		\lim_{\delta_2\rightarrow0}\lim_{\delta_1\rightarrow0}\<\eta_{\delta_1,\delta_2}G, w\>_{t,x}=\int_{0}^{\infty}\int_{M}G(t,x)w(t,x)dxdt=\int_{0}^{\infty}\<G(t),w(t)\>_xdt.
	\end{equation}
	Noticing that $\int_{0}^{\delta_1}|\eta_{\delta_1}'(t)|\lesssim1$ and the continuity of $\<\nabla\eta_{\delta_2}\cdot\nabla u(t), w(t)\>_x$ and $\<(1-\tilde{\eta}_{\delta_2})\partial_tu(t), w(t)\>_x$, we obtain that 
	\begin{equation}\label{eq-proof-prop-regularity-delta1}
		\begin{aligned}
			&\quad\lim_{\delta_1\rightarrow0} \left[\int_{0}^{\delta_1}(-\eta_{\delta_1}'(t))[\<(1-\tilde{\eta}_{\delta_2})\partial_tu(t), w(t)\>_x-\<(1-\tilde{\eta}_{\delta_2})g,w(0)\>_x] dt\right.\\
			&\quad\quad\left. +\int_{0}^{\infty}(1-\eta_{\delta_1}(t)) \<\nabla\tilde{\eta}_{\delta_2}\cdot\nabla u(t), w(t)\>_x dt\right]\\
			&=\int_{0}^{\infty}\<\nabla\tilde{\eta}_{\delta_2}\cdot\nabla u(t), w(t)\>_x dt.
		\end{aligned}
	\end{equation}
	Since $\mathcal{K}\neq\emptyset$, the measure of the support of $\nabla\tilde{\eta}_{\delta_2}$ is of the size $\delta_2$ and
	\begin{equation*}
			\|\nabla\tilde{\eta}_{\delta_2}\|_\infty\leq \|\chi_{\{y:\text{dist}(y,\mathcal{K})\leq\frac{3}{2}\delta_2\}}\|_1\|\nabla\eta_{\frac{\delta_2}{2}}\|_{\infty}\lesssim \frac{1}{\delta_2}.
	\end{equation*} 
	Hence, by $|w(t,x)|\lesssim\delta_2$ for $x\in \supp \nabla\tilde{\eta}_{\delta_2}$,
	\begin{equation}\label{eq-proof-prop-regularity-delta2}
		|\<\nabla\tilde{\eta}_{\delta_2}\cdot\nabla u(t), w(t)\>_x|\leq\|\nabla u(t)\|_{L_x^2}\|\nabla\tilde{\eta}_{\delta_2}w(t)\|_{L_x^2} \rightarrow 0, \ \delta_2\rightarrow0,
	\end{equation}
	uniformly in $t$. Combining \eqref{eq-proof-prop-regularity-inte-1st}, \eqref{eq-proof-prop-regularity-limit}, \eqref{eq-proof-prop-regularity-delta1-delta2}, \eqref{eq-proof-prop-regularity-Gw}, \eqref{eq-proof-prop-regularity-delta1}, and \eqref{eq-proof-prop-regularity-delta2}, we obtain that
	\begin{equation}\label{eq-proof-prop-regularity-inte-finial}
		\begin{aligned}
			&\quad\int_{0}^{\infty}\<u(t),\Box w(t)\>_x dt\\
			&=-\<f,\partial_t w(0)\>_x+\<g,w(0)\>_x+\int_{0}^{\infty}\<G(t),w(t)\>_xdt.
		\end{aligned}
	\end{equation}
	
	Next, we replace $w$ in \eqref{eq-proof-prop-regularity-inte-finial} by $\partial_tw$ and find that
	\begin{equation*}
		\begin{aligned}
			&\quad \int_{0}^{\infty}\<\partial_tu(t),\Box w(t)\>_x dt\\
			&=\<\Delta f, w(0)\>_x-\<g,\partial_tw(0)\>_x-\int_{0}^{\infty}\<G(t),\partial_tw(t)\>_xdt.
		\end{aligned}
	\end{equation*}
	Using the limitation argument above, by $G\in C_t((H^{-1}(M))_{w*})$ and $\partial_t G\in LE^*+L_t^1L_x^2$, one has
	\begin{equation*}
		\begin{aligned}
			&\quad \int_{0}^{\infty}\<\partial_tu(t),\Box w(t)\>_x dt\\
			&=\<\Delta f+G(0  ), w(0)\>_x-\<g,\partial_tw(0)\>_x+\int_{0}^{\infty}\int_{M}\partial_tG(t,x)w(t,x)dxdt.
		\end{aligned}
	\end{equation*}
	According to the compatibility condition, $g\in \dot{H}_D^1$ and $\Delta f(x)+G(0,x)\in L_x^2$, there exists a solution $v\in C_b([0,\infty);\dot{H}_D^1(M)$, $\partial_t v\in C_b([0,\infty);L_x^2(M)$ to the equation 
	\begin{equation*}
		\left\{\begin{aligned}
			&\Box v(t,x)=\partial_tG(t,x)&&, \ (t,x)\in (0,T)\times M, \ \\
			&v(t,x)=0&&, \ x\in \partial M,  \ t>0 , \\
			&v(0,x)=g(x), \ \partial_tv(0,x)=\Delta f(x)+G(0,x) &&, \ x\in M,
		\end{aligned}
		\right.
	\end{equation*}
	which implies
	\begin{equation*}
		\begin{aligned}
			&\quad \int_{0}^{\infty}\<v(t),\Box w(t)\>_x dt\\
			&=\<\Delta f+G(0  ), w(0)\>_x-\<g,\partial_tw(0)\>_x+\int_{0}^{\infty}\int_{M}\partial_tG(t,x)w(t,x)dxdt.
		\end{aligned}
	\end{equation*}
	Therefore, for all $w\in C_c^\infty(\overline{(0,\infty)\times M})$,
	\begin{equation}\label{eq-distribution-equation}
		\int_{0}^{\infty}\<\partial_tu(t)-v(t),\Box w(t)\>_x dt=0.
	\end{equation}
	
	As long as for any $h\in C_c^\infty((0,\infty)\times M)$ we can find a $w_h\in C_c^\infty(\overline{(0,\infty)\times M})$ such that $\Box w_h=h$, it can be deduce by \eqref{eq-distribution-equation} that $\partial_tu=v$.
	Because $h$ has compact support, we can choose $\bar{T}$ and $\bar{R}$ such that $\supp h\subset (0,\bar{T})\times B_{\bar{R}-4\bar{T}}$ and $\mathcal{K}\subset B_{\bar{R}-4\bar{T}}$.
	Consider the backward Dirichlet-wave equation on $B_{\bar{R}}\setminus\mathcal{K}$
	\begin{equation}\label{eq-prop-regularity-Dirichlet-wave-BminusK}
		\left\{\begin{aligned}
			&\Box \tilde{w}(t,x)=h(t,x)&&, \ (t,x)\in (0,\bar{T})\times (B_{\bar{R}}\setminus\mathcal{K}), \ \\
			&\tilde{w}(t,x)=0&&, \ x\in \partial (B_{\bar{R}}\setminus\mathcal{K}),  \ t\in [0,\bar{T}] , \\
			&\tilde{w}(\bar{T},x)=0, \ \partial_tv(\bar{T},x)=0 &&, \ x\in B_{\bar{R}}\setminus\mathcal{K}.
		\end{aligned}
		\right.
	\end{equation}
	Since we have assumed that the obstacle $\mathcal{K}$ is smooth, the problem \eqref{eq-prop-regularity-Dirichlet-wave-BminusK} has a solution $w\in C^\infty(\overline{(0,\bar{T})\times(B_{\bar{R}}\setminus\mathcal{K})})$.
	By finite speed of propagation, we deduce that $w\equiv0$ near time $\bar{T}$ and $w(t,x)=0$, $(t,x)\in [0,\bar{T}]\times\{y:\bar{R}-2.5\bar{T}\leq|y|\leq\bar{R}-1.5\bar{T}\}$. Hence, let
	\begin{equation*}
		w_h(t,x)=\left\{\begin{aligned}
			&\tilde{w}(t,x) &&, \ (t,x)\in [0,\bar{T}]\times B_{\bar{R}-1.5\bar{T}},\\
			&0&&,\ \text{else},
		\end{aligned}
		\right.
	\end{equation*}
	which satisfies the requirement $\Box w_h=h$.
	
	Having $\partial_t u\in C_b([0,\infty);H_D^1(M))$, $\partial_t^2 u\in C_b([0,\infty);L_x^2(M))$, we can estimate $Yu$.
	For any smooth cut-off function $\sigma$ of $\mathcal{K}$, $Yu=Y[\sigma u]+Y[(1-\sigma)u]$.
	As $u, \partial_t u\in C_b([0,\infty);\dot{H}_D^1(M))$, $\|u\|_{LE}<\infty$, and $G\in C_bL_{loc}^2$, it follows that by Poincar\'e inequalities and Sobolev embedding
	\begin{equation*}
		\|\sigma u\|_{L_x^2}\lesssim\|\nabla\sigma u\|_{L_x^2}+\|\sigma \nabla u\|_{L_x^2}\lesssim \|\nabla\sigma \|_{L_x^n}\|u \|_{L_x^{q}}+\|u\|_{\dot{H}^1}\lesssim\|u\|_{\dot{H}^1}, \ \frac{1}{q}+\frac{1}{n}=\half,
	\end{equation*}
	and, then
	\begin{gather*}
		Y[\sigma u]=Y[\sigma]u+\sigma Yu\in C_b\cap L_t^2(L_x^2), \\
		\partial_tY[\sigma u]=Y[\sigma]\partial_tu+\sigma Y\partial_tu\in C_b\cap L_t^2(L_x^2),\\
		\Delta [\sigma u]=\sigma(\partial_t^2u-G)+2\nabla\sigma\cdot\nabla u+[\Delta\sigma] u\in C_bL_x^2.
	\end{gather*}
	Since $\sigma G\in C_bL_x^2$ and $G=G_1+G_2$ with $G_1\in LE^*$ and $G_2\in L_t^1L_x^2$, one has $\sigma G_1\in L_t^2L_x^2$ and any bounded part of $\|\sigma G_2(t)\|_{L_x^2}$ belongs to $L_t^2$.
	On the other hand, $\sigma G_2=\sigma G-\sigma G_1$ implies 
	\begin{equation*}
		\{t:\|\sigma G_2(t)\|_{L_x^2}> 2\|\sigma G\|_{L_t^\infty L_x^2}\}\subset \{t:\|\sigma G_1(t)\|_{L_x^2}> \|\sigma G\|_{L_t^\infty L_x^2}\}
	\end{equation*}
	and on the set $\{t:\|\sigma G_2(t)\|_{L_x^2}> 2\|\sigma G\|_{L_t^\infty L_x^2}\}$
	\begin{equation*}
		\|\sigma G_2(t)\|_{L_x^2}\leq 2\|\sigma G_1(t)\|_{L_x^2}\in L_t^2.
	\end{equation*}
	Thus, we deduce that $G\in C_b\cap L_t^2(L^2)$ and $\Delta [\sigma u]\in C_b\cap L_t^2(L_x^2)$. Further, by elliptic estimates,
	$Y[\sigma u]\in C_b\cap L_t^2(H^1)$. 
	Meanwhile, modifying $(1-\sigma)u$ by $(1-\sigma(\cdot))u(t,\cdot)\ast\eta_{\delta}(|\cdot|)$ and by Newton-Leibniz formula, one obtain that $(1-\sigma)u\in C(L_x^2)$, hence $u\in C(H_D^1)$.
	
	Once again, using the above argument, we have
	\begin{equation*}
		\begin{aligned}
			&\quad\int_{0}^{\infty}\<Y[(1-\sigma)u](t),\Box w(t)\>_x dt\\
			&=-\<Y[(1-\sigma)f],\partial_t w(0)\>_x+\<Y[(1-\sigma)g],w(0)\>_x\\
			&\quad+\int_{0}^{\infty}\<Y[(1-\sigma)G](t)+2Y[\nabla\sigma\cdot\nabla u](t)+Y[\Delta\sigma u](t),w(t)\>_xdt.
		\end{aligned}
	\end{equation*}
	Thus, the inhomogeneous term satisfies
	\begin{equation*}
		Y[(1-\sigma)G]+2Y[\nabla\sigma\cdot\nabla u]+Y[\Delta\sigma u]\in LE^*+L_t^1L_x^2,
	\end{equation*}
	which demonstrates that $Y[(1-\sigma)u]\in C_b(\dot{H}_D^1)$ and $\partial_tY[(1-\sigma)u]\in C_b(L_x^2)$.
	In conclusion, we have $u\in C([0,\infty);H_D^{2}(M))$ and
	\begin{gather*}
		\partial_t^i u\in C_b([0,\infty);H_D^{2-i}(M)), \ 1\leq i \leq 2, \\
		\partial_t^jY^\alpha u\in C_b([0,\infty);H^{2-j-|\alpha|}(M)),  \ 1\leq|\alpha|\leq \min\{2-j,1\}.\qedhere
	\end{gather*}
\end{proof}

\section*{Acknowledgments}
The author was supported by China Scholarship Council (No. 202406320284).
The author thanks Professor Chengbo Wang (School of Mathematical Sciences, Zhejiang University) and Doctor Xiaoran Zhang (Beijing Institute of Mathematical Sciences and Applications) for some discussions in the preparation of the paper.

	\bibliographystyle{plain}
\bibliography{critical_reference}

\end{document}